\documentclass[]{article}
\usepackage{amsthm}
\usepackage{amsmath}
\usepackage{ amssymb }
\usepackage{amssymb}
\usepackage{dsfont}
\usepackage{ bbold }
\usepackage{multirow}
\usepackage{fullpage}
\usepackage{pdfpages}
\usepackage{subcaption}
\usepackage{color}
\usepackage[round]{natbib} 

\usepackage[english]{babel}

\newcommand{\R}{\mathbf{R}}
\newcommand{\x}{\mathbf{x}}
\newcommand{\tx}{\tilde{\mathbf{x}}}
\newcommand{\f}{\mathbf{f}}
\newcommand{\xx}{\mathbb{x}}
\newcommand{\dmat}{\mathbf{d}}
\newcommand{\tht}{{\scriptscriptstyle{HT}}}
\newcommand{\tas}{{\scriptscriptstyle{AS}}}
\newcommand{\mse}{{\scriptscriptstyle{M}}}

\newcommand{\tl}{{\scriptscriptstyle{L}}}
\newcommand{\tm}{{\scriptscriptstyle{M}}}
\newcommand{\tr}{{\scriptscriptstyle{R}}}

\newcommand{\matc}{\mathbf{c}}
\newcommand{\matp}{\mathbf{p}}
\newcommand{\m}{\mathbf{m}}

\newcommand{\e}{\mathbf{e}}
\newcommand{\colsxx}{l}  
\newcommand{\I}{\mathbf{i}}
\newcommand{\E}{\text{E}}
\newcommand{\V}{\text{V}}
\newcommand{\hththt}{3HT!}
\newcommand{\rara}{2R!}
\newcommand{\bfI}{\text{\textnormal{I}}}

\newcommand{\thththt}{\scriptscriptstyle{\emph{\textbf{3}}}\scriptscriptstyle{HT} \hspace{-.2mm}{\scriptscriptstyle \emph{\textbf{!}}}}
\newcommand{\trara}{\scriptscriptstyle{\emph{\textbf{2}}} \scriptscriptstyle{R} \hspace{-.2mm}{\scriptscriptstyle \emph{\textbf{!}}}}
\newcommand{\II}{{\scriptscriptstyle \text{\textnormal{I\hspace{-0.0mm}I}}}}
\newcommand{\sI}{{\scriptscriptstyle \text{\textnormal{I}}}}

\newcommand{\var}{\text{\textnormal{Var}}}
\newcommand{\cov}{\text{\textnormal{Cov}}}

\newcommand{\bpi}{\boldsymbol{\pi}}

\newcommand{\0}{\mathbf{0}}

\newcommand{\dipi}{\boldsymbol{\pi}^{-1}}

\newcommand{\diR}{\mathbf{R}}

\usepackage{array}
\newcolumntype{L}[1]{>{\arraybackslash}p{#1}}
\newcolumntype{C}[1]{>{\centering \arraybackslash}p{#1}}

\usepackage{etex,etoolbox}
\usepackage{amsthm,amssymb}
\usepackage{thmtools}
\usepackage{environ}

\makeatletter
\providecommand{\@fourthoffour}[4]{#4}
\newcommand\fixstatement[2][\proofname\space of]{%
	\ifcsname thmt@original@#2\endcsname
	\AtEndEnvironment{#2}{%
		\xdef\pat@label{\expandafter\expandafter\expandafter
			\@fourthoffour\csname thmt@original@#2\endcsname\space\@currentlabel}%
		\xdef\pat@proofof{\@nameuse{pat@proofof@#2}}%
	}%
	\else
	\AtEndEnvironment{#2}{%
		\xdef\pat@label{\expandafter\expandafter\expandafter
			\@fourthoffour\csname #1\endcsname\space\@currentlabel}%
		\xdef\pat@proofof{\@nameuse{pat@proofof@#2}}%
	}%
	\fi
	\@namedef{pat@proofof@#2}{#1}%
}

\globtoksblk\prooftoks{1000}
\newcounter{proofcount}

\NewEnviron{proofatend}{%
	\edef\next{%
		\noexpand\begin{proof}[\pat@proofof\space\pat@label]%
			\unexpanded\expandafter{\BODY}}%
		\global\toks\numexpr\prooftoks+\value{proofcount}\relax=\expandafter{\next\end{proof}}
	\stepcounter{proofcount}}

\def\printproofs{%
	\count@=\z@
	\loop
	\the\toks\numexpr\prooftoks+\count@\relax
	\ifnum\count@<\value{proofcount}%
	\advance\count@\@ne
	\repeat}
\makeatother

\usepackage[affil-it]{authblk}

\declaretheorem[style=plain,name=Theorem, numberwithin=section]{theorem}

\newtheorem{corollary}{Corollary}[theorem]
\newtheorem{lemma}[theorem]{Lemma}

\newtheorem{algorithm}[theorem]{Algorithm}
\newtheorem{remark}{Remark}
\newtheorem{definition}[theorem]{Definition}
\fixstatement[Proof of]{theorem}
\fixstatement[Proof of]{lemma}

\begin{document}

	\title{A Unified Theory of Regression Adjustment \\ for Design-based Inference \\ 
		\small \vspace{5mm} WORKING PAPER}

	\author{Joel A. Middleton{\footnote{
				Charles and Louise Travers Department of Political Science, \textit{University of California, Berkeley.} \\ \emph{{email:} joel.middleton@gmail.com}}} }
	
\maketitle

\begin{abstract}
Under the Neyman causal model, it is well-known that OLS with treatment-by-covariate interactions cannot harm asymptotic precision of estimated treatment effects in completely randomized experiments. But do such guarantees extend to experiments with more complex designs? This paper proposes a general framework for addressing this question and defines a class of generalized regression estimators that are applicable to experiments of any design. The class subsumes common estimators (e.g., OLS). Within that class, two novel estimators are proposed that are applicable to arbitrary designs and asymptotically optimal. The first is composed of three Horvitz-Thompson estimators. The second recursively applies the principle of generalized regression estimation to obtain regression-adjusted regression adjustment. Additionally, variance bounds are derived that are tighter than those existing in the literature for arbitrary designs. Finally, a simulation study illustrates the potential for MSE improvements.

\end{abstract}

\section{Introduction}


In the analysis of randomized experiments, regression adjustment is common. Even though its classical assumptions are not justified under Neyman's (1923) causal model \citep[see][]{freedman08a,freedman08b}, regression's finite population and asymptotic properties may be nonetheless defensible. For completely randomized experiments for example, \cite{lin} demonstrates that asymptotic precision can not be harmed by ordinary least squares (OLS) regression if treatment-by-covariate interactions are included in the specification. And inference with heteroskedastic-consistent standard errors is asymptotically conservative.

Meanwhile, regression-type adjustments have been recommended for a variety of designs other than complete randomization. Within the framework of the Neyman-Rubin causal model, authors have considered adjustments for cluster-randomized designs \citep{hansenbowers, middletonaronow}, block-randomized designs \citep{atheyimbens}, arbitrary designs \citep{aronowmiddleton15}, two-stage designs with interference \citep{bassefeller17, sinclairetal}, arbitrary/complex designs with interference \citep{aronowsamii17}, and factorial designs \citep{lu}. \cite{atheyimbens} consider block-, cluster- and pair-randomized designs. 

With the notable exception of results for completely randomized designs \citep{lin, bloniarzetal}, however, few recommendations for practice have relied on guarantees of precision gains or claims of optimality under the Neyman-Rubin model, asymptotic or otherwise. Instead, justifications have tended to be heuristic or have relied on ``model assisted" arguments. And while heuristic arguments that regression should help more often than it hurts may be compelling, some amount of analysis is warranted to address concerns about regression adjustments for arbitrary designs of the sort originally raised by \cite{freedman08a, freedman08b}.

As such, a general framework for analysis of regression's properties for any design may be in order. This paper makes several contributions in that direction. First, a proposed framework allows for any design so that expressions can be maximally general. The re-representation of notation will attempt to make derivations compact and straightforward where they would otherwise be difficult. The benefit will be realized especially when expressions for variance are presented and when they are manipulated for the purpose of deriving expressions for optimal regression adjustment. Moreover, a benefit is realized when variance bounding is defined, allowing bounds to be derived that are tighter than those existing in the literature for arbitrary designs.

Second, the proposed class of generalized regression estimators, analogous to estimators in the sampling literature, are applicable to arbitrary designs. It will be demonstrated that the class subsumes common regression practice. Future research may take as a starting point variations on this class, but the proposed class is conceptually useful.


Third, the paper proposes two asymptotically optimal regression estimators that can be applied to arbitrary designs. To develop them, the two key estimation principles presented in the paper, namely, the Horvitz-Thompson principle (inverse propensity of treatment weighting) and the generalized regression principle, are applied recursively. This yields one asymptotically optimal estimator consisting of three Horvitz-Thompson estimators and another that is {regression-adjusted} regression adjustment. While the proposed estimators may have MSE higher than OLS in small samples, future work could introduce refinements.\footnote{However, for the sake of conceptual clarity, this paper seeks to limit the number of estimation principles introduced to the aforementioned two.}  Moreover, the latter may be applicable to larger experiments and may be uniquely optimal for very complex designs. An application could be, for example, experiments in networks where ``exposure" conditions are determined by a complex network structure \citep[e.g.,][]{aronowsamii17}. Sufficiently-large cluster-randomized experiments may be another application. 

Fourth, the paper illustrates how the proposed framework might be used for analysis of specific designs, in particular, completely randomized and cluster randomized designs. In the case of completely randomized designs, results are known \citep[see][]{lin}. However, the proofs are novel, and, as such, they serve as a bridge between existing results and the proposed framework. Moreover, results provided may be useful in the analysis of designs not considered here. 

\subsection{Plan of the paper}

In Section \ref{section.framework} the overall framework is presented. It will establish the context and notation for the discussion of regression adjustment.  In Section \ref{section.regression} a generalized regression estimator is introduced and its connection to regression as commonly used is clarified.  In Section \ref{section.Optimal} two asymptotically optimal estimators that can be applied to any design are proposed. Section $\ref{section.var.est}$ develops variance estimation for HT and generalized regression estimators.  In two subsequent sections, results for specific designs are derived as special cases within the established framework. The specific designs considered are complete randomization and cluster randomization.\footnote{A few comments on block randomization are given in the discussion.} Section \ref{section.sims} uses simulation to illustrate the potential for reductions in MSE using asymptotically optimal regression adjustment.

\section{Framework}\label{section.framework}

This section establishes the overall framework which will be necessary to develop generalized regression adjustment for arbitrary designs.\footnote{Throughout, it should be understood that ``arbitrary designs," is shorthand for ``pretty much any design within reasonable limits."  Limitations include that the design must be identified, i.e., every unit must have some chance of being in treatment and some chance of being in control. Additional limitations are required to ensure asymptotic properties. For example, a design where the treatment group has a fixed number of units as $n\rightarrow \infty$ is outside these limits. Within those limits, the framework is as general as possible.} In the next subsection, average treatment effect (ATE), the causal quantity of interest, is defined in the context of the Neyman-Rubin causal model (NRCM). A two-arm experiment is assumed. While generalization to multiple arms is straightforward, the added notation distracts from the core developments. Throughout the paper, footnotes will attempt to highlight insights about multi-arm extensions.

In the second subsection, the HT estimator is formally introduced.  The HT estimator serves as foundational estimator to which regression adjustments will be made. The estimator has the virtue of being unbiased for the ATE for any identified design, a property which implies that, asymptotically speaking, consistency requires only that its variance goes to zero. By contrast, an alternative estimator such as the difference-of-means could be badly biased and lack consistency in designs where assignment probabilities are unequal.  Moreover, limitations of the HT estimator, namely imprecision and a lack of location invariance, will be addressed by the regression adjustment.\footnote{Another obvious alternative would be to take Hajek estimator as foundational, but the random denominator introduces an inelegance that impedes the illustration of principles. It is also provably true that (under regularity conditions) generalized regression estimators that are HT-based obtain the equal asymptotic variance to Hajek-based counterparts. That said, for smaller samples, the development of a Hajek-based generalized regression estimator might be a worthwhile refinement. However, it is beyond the scope of this paper.} 
Importantly, the variance of the HT estimator will be discussed.  As it will be demonstrated later, this variance expression is directly relevant to variance expressions for the generalized regression estimator. The topic of variance {\it estimation} will be postponed until Section $\ref{section.var.est}$, when an overall strategy can be presented for both HT estimators and generalized regression estimators. 

\subsection{The Neyman Causal Model for treatment-control experiments}

Consider the Neyman-Rubin causal model (NRCM) and imagine a two-arm experiment. For convenience, refer to it as a treatment-control experiment, with one arm being referred to as the treatment and the other control. For the $i^{th}$ unit of the (finite) study population of $n$ units, there are two (fixed) potential outcomes: $y_{0i}$ and $y_{1i}$. Which of $i$'s outcomes is reveled is determined only by $i$'s treatment assignment indicator, $R_{1i}$, which is random, the only random component in the NCM. The researcher observes for each unit the outcome $\left(R_{1i},\hspace{1mm} R_{1i}y_{1i}, \hspace{1mm}(1-R_{1i})y_{0i}, \hspace{1mm} x_i \right)$, where $x_i$ is an additional vector of $k$ covariates which is observed for every unit irrespective of assignment. The parameter of interest is the average treatment effect (ATE), which can be written
\begin{align*}
\delta := & n^{-1}\sum_i \left(y_{1i}-y_{0i}\right).
\end{align*}
The ``fundamental problem of causal inference" \citep{holland} is that only one of the two potential outcomes can be observed for each unit.

For ease of derivations, it makes sense to re-represent the problem in a nonstandard way starting with the potential outcomes. First, represent the entire schedule of potential outcomes as the vector
\begin{align*}
y := & 
\left[ \begin{matrix}
-y_{01} \\ -y_{02} \\ \vdots 
\\ -y_{0n} \\ y_{11}\\ y_{12}\\ \vdots 
\\ y_{1n}
\end{matrix}\right]
\end{align*}
and note that $y$ has length $2n$ and that the control potential outcomes are multiplied by $-1$. This allows the ATE to be represented equivalently as 
\begin{align}\label{ate}
\delta = & n^{-1} {1}'_{{\scriptscriptstyle 2n}} y
\end{align}
where ${1}_{{\scriptscriptstyle 2n}}$ is a column vector of $2n$ ones.\footnote{Alternatively, one could represent a contrast matrix explicitly, but at a notational cost. For example, one could define $\mathbf{c}=\left[\begin{matrix}
-\I & 0 \\ 0 & \I
\end{matrix} \right]$, where $\I$ is an $n \times n$ identity matrix, and write $\mathbf{c} y$ throughout instead of using the more compact $y$ as defined above.  Explicitly representing contrast matrices is useful when generalizing these results to more than two treatment arms.}  This ``stacked" representation of the potential outcomes allows for mathematical manipulation of a single entity, where the canonical frameworks requires manipulation of two entities. The benefits become evident in variance expressions and as covariates are added to the picture. 


If, as above, $R_{1i}$ is the (random) treatment assignment indicator for the $i^{th}$ individual and defining the control assignment indicator as $R_{0i}:=1-R_{1i}$, then the $2n\times2n$ matrix
\begin{align*}
\R :=&
 \left[ \begin{matrix}
R_{01} \\ & R_{02} \\ & & \ddots \\& & &  R_{0n} \\ & & & & R_{11} \\ & & & & & R_{12} \\ & & & & & & \ddots \\ & & & & & & & R_{1n} 
\end{matrix}\right], \hspace{2mm}
\end{align*}
encodes on its diagonal which of the elements of $y$ is observed. In matrix form, the researcher can be said to observe $\R$ and $\R y$, as well as an $n \times k$ matrix of covariates, $\x$, which will be assumed to be zero-centered. Zero-centering will make straightforward the interpretation of intercept terms in regression coefficients, but will otherwise not affect the analysis.

\subsection{The Horvitz-Thompson Estimator}\label{section.ht}



As an experiment, the researcher controls the mechanism that determines the distribution of $\R$.  As such, the researcher is also able to compute quantities such as the probability of treatment assignment, $\pi_{1i} := \E[R_{1i}]$, and the probability of assignment to control, $\pi_{0i} := \E[R_{0i}]$ (either analytically or numerically to arbitrary precision). A $2n\times2n$ diagonal matrix of assignment probabilities can be written
\begin{align*}
\bpi := 
\left[ \begin{matrix}
\pi_{01} \\ & \pi_{02} \\ & & \ddots 
\\ & & & \pi_{0n} \\ & & & & \pi_{11} \\ & & & & & \pi_{12} \\& & & & & &  \ddots 
\\ & & & & & & & \pi_{1n}
\end{matrix}\right].
\end{align*}
Note that by definition $\bpi=\E[\R]$.

Next, the HT estimator of the average treatment effect can be defined as
\begin{align}\label{ht.est}
\widehat{\delta}^\tht := n^{-1} 1'_{{\scriptscriptstyle 2n}}  \bpi^{-1} \R y.
\end{align}
where $1_{{\scriptscriptstyle 2n}}$ is, as above, a column vector of $2n$ ones.\footnote{Equivalently, one could have written 
	\begin{align*}
	\widehat{\delta}^\tht = & n^{-1} 1_{\scriptscriptstyle n}' \bpi_1^{-1} \R_1 y_1- n^{-1} 1_{\scriptscriptstyle n}' \bpi_0^{-1} \R_0  y_0,
	\end{align*}
	an expression that has the pedagogical benefit of clarifying that the estimator is a difference between two HT estimators, one for each treatment arm. However, the more compact convention above streamlines.}{\textsuperscript,}\footnote{To the extent possible, an attempt was made to adhere to the following notational conventions: upper case signifies random variables (e.g., $R_{1i}, \R$); bold signifies a matrix (e.g., $\R, \dmat$); non-subscripted letters are vectors (e.g., $y$) or sometimes scalars (e.g., $n$, $k$, $c$); a subscript of 0 or 1 on a letter typically means the subvector associated with control or treatment outcomes, respectively, (e.g., $y_0$, $y_1$); letters with subscripts of $i, j, k$ or $l$ typically mean scalar elements of a vector or matrix (e.g., $y_{1i}$, $\pi_{0i}$, $\pi_{1i1j}$); and Greek typically signifies a quantity of interest or descriptive summary of a population characteristic (e.g., $\delta$).  An exception to this last rule, for reasons of tradition, is the use of $\pi$ for assignment probability which is not a quantity of interest but a parameter set by the design.}  
The diagonal matrix $\bpi^{-1}$ does the work of weighting the potential outcomes inversely proportional to the probability of being observed.  As such, the HT estimator has the virtue of being unbiased for any identified design, i.e., designs in which $0<\pi_{1i}<1$ for all $i$, because $\E[\R]=\bpi$.

An HT estimator is similar to an inverse propensity of treatment weighted (IPTW) estimator, but the ``propensity score" is known in this setting by way of knowledge of the design of the experiment. Also, the estimator can sometimes be equivalent to the difference-of-means, such as in completely randomized designs.



\subsection{Variance of the Horvitz-Thompson estimator}


To express the variance of a HT estimator of the average treatment effect, first note that in equation (\ref{ht.est}), $y$ can be seen as coefficients on the random vector $1_{\scriptscriptstyle 2n}' \bpi ^{-1} \R $.  Thus, if one defines the $2n\times2n$ ``design" matrix
\begin{align}
\dmat:=\V \left(1_{\scriptscriptstyle 2n}' \bpi ^{-1} \R \right)
\end{align}
where $\V \left(.\right)$ represents variance-covariance, then the variance of the HT estimator of the average treatment effect can be written compactly as
\begin{align}\label{HTvar}
\V\left(\widehat{\delta}^{\tht} \right) = n^{-2} y' \dmat y.
\end{align}
Equivalent, though much more cumbersome expressions are given by \cite{aronowmiddleton15} and \cite{aronowsamii17}.\footnote{For example, \cite{aronowmiddleton15} write the variance of the HT estimator as
	\begin{align*}
	\V\left(\widehat{\delta}^{\tht} \right)=& \sum_i \frac{\pi_{1i}(1-\pi_{1i})}{\pi_{1i}\pi_{1i}}y_{1i}^2+\sum_i\sum_{j \neq i} \frac{\pi_{1i1j}-\pi_{1i}\pi_{1j} }{\pi_{1i}\pi_{1j}}y_{1i}y_{1j} 
	\\& +\sum_i \frac{\pi_{0i}(1-\pi_{0i})}{\pi_{0i}\pi_{0i}}y_{0i}^2+\sum_i\sum_{j \neq i} \frac{\pi_{0i0j}-\pi_{0i}\pi_{0j} }{\pi_{0i}\pi_{0j}}y_{0i}y_{0j}
	\\ & -2\sum_i y_{1i}y_{0i}-2\sum_i\sum_{j \neq i} \frac{\pi_{1i0j}-\pi_{1i}\pi_{0j} }{\pi_{1i}\pi_{0j}}y_{1i}y_{0j}.
	\end{align*} } The compact representation of variance given here is essential to deriving a general expression for asymptotically optimal regression adjustment and results for variance bounds, below. 

Insight into how to construct $\dmat$ in practice may arise from exploring its elements further.  First define the joint probability that units $i$ and $j$ are both included in the treatment arm $\pi_{1i1j} :=\E\left[ R_{1i} R_{1j}\right]$, and note that the probability of assignment to treatment for unit $i$ could be written $\pi_{1i1i}$ or $\pi_{1i}$. Similarly, the joint probability of inclusion in the control group is $\pi_{0i0j} :=\E\left[ R_{0i}R_{0j} \right]$. Moreover, $\pi_{1i0j} :=  \E\left[ R_{1i} R_{0j}\right]$ is the probability that $i$ is in treatment and $j$ is in control, and $\pi_{0i1j} :=  \E\left[ R_{0i} R_{1j}\right]$ is the probability that $i$ is in control and $j$ is in treatment. 

Next, note that $\dmat$ can be partitioned into four $n \times n$ matrices. Write
\begin{align}\label{dmat.quad}
\dmat=\left[\begin{array}{ccc}
\dmat_{00} & \dmat_{01} 
\\ \dmat_{10} &\dmat_{11}\end{array}\right]
\end{align}
where, for example, the matrix $\dmat_{11}$ has $ij$ element $\frac{\pi_{1i1j}-\pi_{1i}\pi_{1j}}{\pi_{1i}\pi_{1j}}$ and the matrix $\dmat_{10}$ has $ij$ element $\frac{\pi_{1i0j}-\pi_{1i}\pi_{0j}}{\pi_{1i} \pi_{0j}} $. Sub-matrices $\dmat_{00}$ and $\dmat_{01}$ are defined analogously. The equivalence in (\ref{dmat.quad}) will be useful below. 

Since not all pairs of potential outcomes can be observed together for various reasons, not all terms in the quadratic in (\ref{HTvar}) are observable.  Beginning with Neyman (1923), this has led to the idea of bounding the variance in the design-based paradigm. A general approach to variance bounds and their estimation for both HT and regression estimators will be described in Section \ref{section.var.est}.
	
\section{Regression}\label{section.regression}

Now that the framework has been established for estimation and variance estimation under the NCM, this section turns to the main subject of the paper, regression adjustment.  

\subsection{Covariate specifications}

To discuss covariate specifications, it helps to move inductively from a familiar example. First, let ${\x}$ be the zero-centered matrix of covariates with $n$ rows and $k$ columns. The covariates themselves are taken as given (i.e., the question of how to arrive at the right covariates, how to transform them, etc., is left to another paper) and fixed (i.e., nonrandom and not affected by treatment). Now consider the analysis practice of regressing observed outcomes on separate intercepts for each treatment arm and on ${\x}$ using OLS. Using the current notational framework, this OLS coefficient estimator can be written
\begin{align}\label{b.est.ols}
	\widehat{b}^{ols}_\sI := \left({\xx_{\sI}}' \R \xx_{\sI} \right)^{-1} {\xx_{\sI}}' \R y
\end{align}
where
\begin{align*}
	{\xx}_\sI := & \left[
	\begin{matrix}
		-1_{\scriptscriptstyle n} & 0_{\scriptscriptstyle n} &  -{\x}
		\\ 0_{\scriptscriptstyle n} & 1_{\scriptscriptstyle n} &  {\x} 
	\end{matrix} \right]
\end{align*}
is a $2n \times (k+2)$ matrix.
Note that (\ref{b.est.ols}) is algebraically equivalent to the canonical OLS formulation that typically writes the estimator in terms of an $n \times (k+2)$ covariate matrix and a vector of observed outcomes that has length $n$. By contrast, however, the present formulation has the advantage that it separately represents the source of randomness ($\R$) and the fixed quantities ($\xx_{\sI}$ and $y$). Note that for convenience in later derivations, the leading column of $\xx_\sI$ is an intercept (constant) associated with the control group and the second column is an intercept associated with the treatment group. With only a slight change in interpretation of the coefficients, this could have instead been specified as a constant and a treatment indicator. Also, note that elements in the first $n$ rows of $\xx_\sI$ are multiplied by $-1$, to mirror the definition of the vector $y$ and thus ensuring that the elements of $\widehat{b}^{ols}_\sI$ have the expected signs. The subscript on matrix $\xx_\sI$ is given to distinguish it from an alternative specification given below, and, in later derivations, in the absence of such a subscript, $\xx$ will be taken to represent any arbitrary covariate specification. See also that $\widehat{b}^{ols}_\sI$ shares the subscript indicating the particular specification of $\xx$. The given specification will be referred to ``specification I" or alternatively the ``common slopes" specification.

By contrast, ``specification II" (or the ``separate slopes" specification) is given by,
\begin{align*}
\xx_\II := & \left[
\begin{matrix}
-1_{\scriptscriptstyle n} &  -\x & 0_{\scriptscriptstyle n} &  0_{\scriptscriptstyle n\times k}
\\ 0_{\scriptscriptstyle n} &   0_{\scriptscriptstyle n\times k} & 1_{\scriptscriptstyle n} & {\x} 
\end{matrix} \right]
\end{align*}
which is equivalent to including interactions between treatment and each covariate in $\x$.  \cite{lin}, for example, recommends this specification as a remedy to Freedman's (2008a,b) critique that for completely randomized designs OLS with specification I can in some cases hurt asymptotic precision. Note that, as in specification I above, there is an intercept for each treatment arm, rather than a specifying a common intercept and a treatment indicator. Again, this convention simplifies some exposition below. It does not affect the properties of estimators discussed.

It has yet to be said just what coefficient estimators, such as $\widehat{b}^{ols}_\sI$ in (\ref{b.est.ols}), estimate. For the time being suffice it to say that researchers will often interpret the difference between intercept coefficients in $\widehat{b}^{ols}_\sI$ as an estimate of the ATE, i.e., $[-1 \hspace{2mm} 1  \hspace{2mm}  0_{\scriptscriptstyle k}'\hspace{1mm}]\hspace{1mm}\widehat{b}^{ols}_\sI$ is often taken to be the ATE estimator. However, a generalized regression estimator can be defined which broadens the class of regression estimators to include those with coefficients that may not necessarily be directly interpretable in this fashion. Freeing regression coefficients from the burden of interpretable elements allow for the derivation of certain optimal estimators of the ATE that are not otherwise obvious. 

\subsection{Defining a class of generalized regression estimators}

Three equivalent forms for the proposed class of generalized regression estimators are given in the definition below. Sampling theorists have the longest history with the idea of generalized regression and the constructions that follow.\footnote{The approach herein differs from the ``GREG" estimator in the sampling literature in some key ways, however. First, their results were derived for the sampling setting rather than the causal inference context. Second, that literature has tended to focus on obtaining $\widehat{b}$ coefficients that are optimal under a model. This paper is fully design-based, so asymptotic optimality is considered from the design-based perspective.  Third, the $\widehat{b}$ coefficients considered in the GREG literature has been typically limited to the class $\widehat{b}^{greg}=\left(\tx \m \R \tx \right)^{-1}\tx \m \R y$ where $\m$ is a diagonal matrix with the $i,i$ entry involving $\pi_{i}$ (similar to WLS with $\bpi^{-1}$ weights) and often an estimate of (model) error variance. By contrast, this paper will propose estimators that have a somewhat different form in order to achieve asymptotic optimality in the design-based framework.} Doubly robust estimators also borrow the form from sampling theory.\footnote{There are three reasons not to refer to the generalized regression estimator as ``doubly robust", even though the latter term may be better known. First, the latter term was preceded by the term ``generalized regression estimator", first coined by the sampling theorists some years before. Moreover, unlike doubly robust estimation which were fashioned for observational studies, in the current framework $\bpi$ is given by the design. Hence, the estimator is not ``doubly robust" conditional on getting one or another set of modeling assumptions is correct; on the contrary, one could say that it is simply ``robust" because the treatment assignment probabilities are given and thus correct by design. Moreover, variance expressions in the doubly robust literature do not account for joint assignment probabilities, and hence, are not useful in the current framework. By contrast, variance expressions derived here lead to asymptotically optimal estimators that would not be conceived of in a tradition that assumes away the essential role that joint assignment probabilities play in variance.} The literature on control functions has apparently reinvented the generalized regression estimator as well.

\begin{definition}[Generalized regression estimators]\label{GenReg} Three equivalent forms for ``generalized regression estimators" of the average treatment effect (ATE) are given by
\begin{subequations}\label{greg}
	\begin{align}
\widehat{\delta}^{\tr} : = & n^{-1} 1'_{\scriptscriptstyle 2n}  (\bpi^{-1}\R y- \bpi^{-1}\R\xx \widehat{b} - \xx \widehat{b}  ) \label{greg.a}
\\ = &\widehat{\delta}^\tht - \widehat{\delta}_\xx^\tht \widehat{b} \label{greg.b}
\\ = & n^{-1} 1'_{\scriptscriptstyle 2n}\bpi^{-1} \R\widehat{u} +n^{-1} 1'_{\scriptscriptstyle 2n} \xx \widehat{b} \label{greg.c}
\end{align}
\end{subequations}
where $\widehat{\delta}^\tht$ is the HT estimator of the ATE, $\widehat{\delta}_\xx^{\tht}:= n^{-1} 1'_{\scriptscriptstyle 2n} \left( \bpi^{-1}\R - \I_{\scriptscriptstyle 2n} \right)\xx$ is a zero-centered vector of HT estimators of the column sums of $\xx$ divided by $n$, and $\widehat{u}:=y-\xx  \widehat{b}$.  Vector $\widehat{b}$ is an arbitrary coefficient estimator. Specific estimators in this class are distinguished by the particular $\widehat{b}$.
\end{definition}

The three forms of the generalized regression estimator in (\ref{greg}) are useful at different times in subsequent derivations.  Form (\ref{greg.a}) is the most disaggregated form. By grouping the terms inside the parentheses in different ways, the latter two forms can be derived.  

To arrive at (\ref{greg.b}), first factor out $\xx \widehat{b}$ from the second and third terms inside the parenthesis in (\ref{greg.a}). Then define the zero-centered HT estimator associated with $\xx$, $\widehat{\delta}_\xx^{\tht}:= n^{-1} 1'_{\scriptscriptstyle 2n} \left( \bpi^{-1}\R - \I_{\scriptscriptstyle 2n} \right)\xx$.  To see that $\widehat{\delta}_\xx^{\tht}$ is zero-centered, simply take its expectation, noting that $\bpi^{-1}\E[\R]=\I_{\scriptscriptstyle 2n}$. This form makes it clear that the generalized regression estimator is just the HT estimator of the ATE minus an adjustment term.  Its compactness will make it useful for deriving an asymptotic result below.

To arrive at (\ref{greg.c}), factor $\bpi^{-1}\R$ out of the first two terms inside the parenthesis in (\ref{greg.a}) and define $\widehat{u}:=y-\xx  \widehat{b}$, essentially a vector of residuals.  $\xx \widehat{b}$ is analogous to a vector of predicted values.  The equivalence shows that the generalized regression estimator can be thought of as an HT estimator of the mean of residuals plus an average of predicted values. In this form, which is also useful for certain variance derivations, some will see the connection to doubly robust estimators (but see footnote 10).

As previously mentioned, particular members of the generalized regression estimator class are determined by the corresponding definitions of $\widehat{b}$. Throughout, a superscript will be added to $\widehat{b}$ to signify a particular estimation method and a subscript will indicate covariate specifications. For example, $\widehat{b}^{ols}_\sI$ would be the ``common slopes" OLS regression as given in (\ref{b.est.ols}). The same superscript and subscript can be added to the corresponding ATE estimator, $\widehat{\delta}^{\tr, ols}_\sI$, as well. The pair ($\widehat{b}^{ols}_\sI$, $\widehat{\delta}^{\tr, ols}_\sI$) can be referred to as ``conjugates". Every unique $\widehat{b}$ implies a conjugate ATE estimator.

While the common use of regression involves interpreting the difference in intercept coefficients in $\widehat{b}$ as the ATE, the class of estimators defined here is broader. As such, one way to look at the utility of the generalized regression estimator is that it prescribes a general method of obtaining an estimate of the ATE from a broader array of specification-estimator-design combinations.\footnote{For example, suppose a researcher had block randomized with unequal assignment probabilities across blocks. Given the design, the first element of the OLS coefficient in ($\ref{b.est.ols}$) is not generally consistent for the ATE, and, hence, it can not be interpreted directly as such. However, (\ref{greg}) gives a general formulation for producing a consistent estimator of the ATE using a broad array of coefficient estimators, even those that are not directly interpretable given the design-specification combination. Consistency is discussed further, below.} This, for example, will allow us to define coefficient vectors with asymptotically optimal conjugates that would not otherwise be obvious (see Section \ref{section.Optimal}). 

In the next subsection, a general condition whereby the difference in intercept coefficients in $\widehat{b}$ will be algebraically equivalent to its conjugate ATE estimator, and hence directly interpretable, is given. This will show that the class of generalized regression estimators subsumes common regression practice of interpreting the difference in intercept coefficients as ATE estimates. Moreover, the result will lead to a few insights that may be familiar, but which all follow nicely from the one theorem.

\subsection{Common uses of regression are subsumed by the class of generalized regression estimators}

Since the most common regression practice is to interpret the difference in intercept terms in coefficient vectors as ATE estimates, it serves to connect that practice to the generalized framework. The following theorem shows that, in common practice, the difference in intercept coefficients is algebraically equivalent to the generalized regression estimator, thus explaining in what sense equation (\ref{greg}) ``generalizes" regression. 


The main thrust of the following theorem is to establish conditions under which the first term in (\ref{greg.c}) will be equal to zero, algebraically speaking.

\begin{theorem}\label{thrm.special} Let $\mathbf{m}$ be any symmetric, positive definite $2n \times 2n$ matrix and $\widehat{b}^\mathbf{m}=\left(\xx' \R' \mathbf{m}^{-1} \R \xx \right)^{-1} \xx' \R' \mathbf{m}^{-1} \R y$ (a class which encompasses GLS, WLS and OLS), then the conjugate generalized regression estimator, $\widehat{\delta}^{\tr,\mathbf{m}}$, is algebraically equivalent to $n^{-1} 1'_{\scriptscriptstyle 2n} \xx  \widehat{b}^{\mathbf{m}}$ if $\exists z$ such that 
	\begin{align}
	\R \xx  z =&(\R\mathbf{m}^{-1}\R)^{(-)}\bpi^{-1}\R {1_{\scriptscriptstyle 2n}}
	\end{align}
where $z$ is some vector of constants that combines the $x$'s, and $(.)^{(-)}$ is the Moore-Penrose generalized inverse.
\end{theorem}
\begin{proof}
	First note that to prove the theorem, we need to show that the above condition implies that the first term in (\ref{greg.c}) equals zero, i.e., that
	\begin{align}\label{cond1}
	  \left(y-\xx  \widehat{b} ^\mathbf{m} \right)' \bpi^{-1}\R 1_{\scriptscriptstyle 2n}=&0.
	\end{align}
	To see when the equality in (\ref{cond1}) will hold, first note that from the definition of $\widehat{b}^\mathbf{m} $ given in the theorem we have
	\begin{align*}
	\left( y-\xx \widehat{b}^\mathbf{m} 
	\right)'  \R' \mathbf{m}^{-1} \R \xx   =0.
	\end{align*}
	Therefore, it must also be the case that
	\begin{align}\label{cond2}
	\left( y-\xx \widehat{b}^\mathbf{m} 
	\right)'  \R' \mathbf{m}^{-1} \R \xx z  =0
	\end{align}
	for any vector $z$ that linearly combines the $x$'s in some way. 
	Comparing the condition given in (\ref{cond1}) to the equality in (\ref{cond2}) we can see that the only need 
	\begin{align*}
	\R \mathbf{m}^{-1}\R \xx  z  =\bpi^{-1} \R {1_{\scriptscriptstyle 2n}}
	\end{align*}
	for some value of $z$.  This is satisfied when there exits a $z$ such that 
		\begin{align*}
		\R \xx  z  =(\R \mathbf{m}^{-1}\R)^{(-)} \bpi^{-1} \R {1_{\scriptscriptstyle 2n}},
		\end{align*}
	completing the proof. 
\end{proof}

\begin{remark}[Algebraic equivalences for OLS in equal-$\pi$ designs]
	For any identified design with equal $\pi_{1i}$ for all $i$ (such as a completely randomized design) when using OLS (i.e., $\m^{-1}$ is an identity matrix), the condition in Theorem \ref{thrm.special} reduces to $\R \xx  z = \bpi^{-1} \R1_{\scriptscriptstyle 2n}$. This is trivially satisfied in specifications with an intercept for each treatment arm (such as specification I and specification II) and for equivalent specifications (such as a common intercept with a treatment indicator). For specification \emph{I}, this means that the generalized regression estimator \emph{$\widehat{\delta}^{\tr,ols}_{\sI}$} is algebraically equivalent to the difference in intercept terms its conjugate coefficient estimator, \emph{$\widehat{b}^{ols}_{\sI}$}.  For specification \emph{II} this means that, if the columns of ${\x}$ have mean zero, then the generalized regression estimator \emph{$\widehat{\delta}^{\tr,ols}_{\II}$} is algebraically equivalent to the difference of intercept terms in its conjugate coefficient estimator, \emph{$\widehat{b}^{ols}_{\II}$}.
\end{remark}

\begin{remark}[Algebraic equivalences for WLS with $\bpi^{-1}$ weights]
	{For any identified design when using WLS with $\bpi^{-1}$ weights (i.e., when $\mathbf{m}^{-1}=\bpi^{-1}$) the condition in Theorem \ref{thrm.special} reduces to $\R \xx  z = \R  1_{2n}$. This is trivially satisfied in specifications with a separate intercept for each treatment arm (such as specification \emph{I} and specification \emph{II}) and for equivalent specifications (such as a common intercept with a treatment indicator). }
	{For specification \emph{I} this means that the generalized regression estimator \emph{$\widehat{\delta}^{\tr,\pi wls}_\sI$} is algebraically equivalent to the difference in intercept terms in the conjugate coefficient estimator, \emph{$\widehat{b}^{\pi wls}_\sI$}.}
	{For specification \emph{II} this means that, if the columns of ${\x}$ have mean zero, then the generalized regression estimator \emph{$\widehat{\delta}^{\tr,\pi wls}_\II$} is algebraically equivalent to the difference in intercept terms its conjugate coefficient estimator, \emph{$\widehat{b}^{\pi wls}_\II$}.}
\end{remark}

\begin{corollary}\label{AddPiToXX}
	One can ensure that the condition in Theorem \ref{thrm.special} holds by including a vector, $v$, in $\xx$ that satisfies $\R v =(\R \mathbf{m}^{-1}\R)^{(-)} \bpi^{-1} \R {1}$.  
\end{corollary}

\begin{remark}
	Corollary \ref{AddPiToXX} shows that for any identified design the condition is satisfied for OLS when the reciprocal of probability of assignment, $\bpi^{-1} 1_{\scriptscriptstyle 2n}$, is included as a covariate in $\xx$. Moreover, including a zero-centered version of $\bpi^{-1} 1_{\scriptscriptstyle 2n}$ in $\xx$ would allow for interpretation of the difference in intercept terms.  
\end{remark}

\begin{remark}
	Note that in spite of Corollary \ref{AddPiToXX}, $(\R\mathbf{m}^{-1}\R )^{(-)}\bpi^{-1}\R1$ could effectively be a random variable if $\m$ is not a diagonal matrix. Thus adding it to the matrix $\xx$ could have unexpected consequences for variance and bias of the ATE estimator.
\end{remark}

Remarks 1-3 show that in some cases it is possible to directly interpret the difference in intercept coefficients as an estimated ATE. These are special cases of the generalized regression estimator in Definition \ref{GenReg}, and these relationships reveal in what sense it ``generalizes" (i.e., subsumes) regression approaches in common use. The generalization, however, will allow for the definition of ATE estimators that obtain asymptotic optimality but for which the difference in intercept terms in the conjugate may not be readily interpretable.

Even when using an estimator-design-specification combination where the condition in Theorem \ref{thrm.special} holds, knowing the point estimate of the ATE is $n^{-1} 1'_{\scriptscriptstyle 2n} \xx  \widehat{b}^{\mathbf{m}}$ can be useful. For example, it leads to the well-known maxim that the difference in intercept terms can only be directly interpreted in specification II if covariates, $\x$, are transformed to have mean zero. However, should the researcher forgo zero-centering, it still algebraically defines the process for arriving at an ATE estimate.

\subsection{Variance of the generalized regression estimator when $\widehat{b}$ is fixed \\ (-and- A post-hoc test of improved precision)}\label{section.fixedb}

An exact expression for the variance of the generalized regression estimator is straightforward when $\widehat{b}$ is a fixed vector of constants, call it $b^{f}$.\footnote{In sampling theory, when the $b$ coefficients are fixed constants the corresponding estimator is called a ``difference estimator".}  In theory, a researcher might obtain this fixed vector through examination of an auxiliary data set, or by way of conjecture, insight or divination. In practice, researchers will likely estimate coefficient values from the data at hand. Nonetheless, the variance of $\widehat{\delta}^{\tr,f}$ (the conjugate of the fixed coefficient $b^f$) is useful to consider for the following reasons. First, the variance expression for $\widehat{\delta}^{\tr,f}$ will help to establish the asymptotic variance expression for generalized regression estimators (see section \ref{section.asymptotic.arguments}). Second, a value of $b^{f}$ that is finite sample optimal is a quantity that a coefficient estimator might target to obtain {\it asymptotic} optimality (see section \ref{section.Optimal}). This section also provides the basis for a test of the null hypothesis that adjustment does not help precision.

\begin{definition}[Fixed-coefficient generalized regression estimators] ``Fixed-coefficient generalized regression estimators" are in a subclass of generalized regression estimators defined in $(\ref{greg})$ where $\widehat{b}=b^f$ and $b^f \in \mathds{R}^{l}$ is a vector of constants.
\end{definition}

\begin{lemma}\label{lemma.finiteVar}
	The finite sample variance of the fixed-coefficient generalized regression estimator, $\widehat{\delta}^{\tr,f}$, with conjugate $b^{f}$ being a fixed constant, is 
	\begin{align}\label{eq.finiteVar}
	\emph{V} \left( \widehat{\delta}^{\tr,f} \right)=n^{-2}u'\dmat u
	\end{align}
	where $u:=y-\xx b^{f}$.
\end{lemma}

\begin{proof}
	To see the result, start with the third form of the generalized regression estimator given in (\ref{greg.c}). Note that when $\widehat{b}=b^f$, a constant vector, the second term in (\ref{greg.c}) is a constant. The first term is recognizable as a HT estimator for the mean of vector $u :=y - \xx b^f$ (i.e., the residual vector), and note that $u$ is fixed, not random, for a given $b^f$. Hence, the exact variance is constructed as in equation (\ref{HTvar}) but with $u$ in place of $y$.
\end{proof}


One question is whether fixed-coefficient generalized regression estimator improves precision over the HT estimator, and how one might be confident of that in practice.  The answer to this question will suggest the basis of a hypothesis test that can be further explored after Section \ref{section.var.est} on variance estimation. To begin to develop the idea of the test, note that ($n^2$ times) the difference in variances can be written
\begin{align*}
	n^2 \V \left( \widehat{\delta}^{\tht} \right) - n^2 \V \left( \widehat{\delta}^{\tr,f} \right)=& y'\dmat y -u'\dmat u
	\\ = & y'\dmat y-\left(y'\dmat y - 2 b^{f'} \xx'\dmat y + b^{f'} \xx'\dmat \xx b^f\right)
	\\ = & 2 b^{f'} \xx'\dmat y - b^{f'} \xx'\dmat\xx b^f
	\\ = & 2 b^{f'} \xx'\dmat (y- \xx b^f)  +b^{f'} \xx'\dmat\xx b^f
	\\ = & 2 b^{f'} \xx'\dmat u  +b^{f'} \xx'\dmat\xx b^f
\end{align*}
where the vector $u:=y-\xx b^f$ is not observed for every unit, so that the quantity cannot be observed.  However, an estimator can be proposed by defining length-$2n$ column vector $v=(2b^{f'}\xx'\dmat \text{diag}(u))'$ and testing whether $\widehat{\delta}^{\tht}_v:=n^{-1}1_{\scriptscriptstyle 2n}\bpi^{-1}\R v$ is greater than $-n^{-1}b^{f'} \xx'\dmat\xx b^f$.  Since $\widehat{\delta}^{\tht}_v$ is just an HT estimator, the same machinery that will be developed in Section \ref{section.var.est} for conservative variance estimation can be applied to it and conservative inference can follow. 

Note that the proposed method tests for the difference-of-variances, which is identified, even though the variances are not themselves identified. By contrast, a direct comparison of variance estimators based on those proposed in Section \ref{section.var.est} would actually be a comparison of variance bounds, and, hence, less relevant. 

An analogous test for generalized regression estimators with coefficients estimated from the data could also be developed. These tests could help reassure analysts with concerns that regression adjustment can sometimes hurt asymptotic precision \citep{freedman08a, freedman08b}. Of course, estimation decisions should be set ahead of time in pre-analysis plans, and such a test should only be used in retrospect.  But such a test could still be useful for decision making, say, in a review of past studies to help a researcher determine whether to use regression adjustment in a future study.

\subsection{An asymptotic argument}\label{section.asymptotic.arguments}

In this section conditions for generalized regression estimators to be asymptotically unbiased, consistent, and asymptotically normal are given. The conditions given are somewhat high-level because greater specificity is difficult without first limiting asymptotic analysis to a particular design (e.g., complete randomization, cluster randomization, block randomization, etc.) and perhaps being more specific about the class of coefficient estimators (e.g., OLS, WLS, etc.).

That said, an important conclusion in this subsection is that, asymptotically speaking, a key consideration is whether a sequence of designs and finite populations are such that HT estimators are root-$n$ consistent and asymptotically normal. If they are, then the coefficient, $\widehat{b}$, need only converge in probability. On the one hand, for sufficiently large $n$ this provides a certain amount of freedom to choose coefficient estimators whose asymptotic normality or rate of convergence is uncertain. On the other hand, the reliance on the properties of HT estimators should be reassuring because they are well-studied, and their asymptotic properties are worked out under a variety of designs.

\vspace{2mm} 
\noindent \textbf{Assumptions:} 
\begin{enumerate}
	\item[1.] (Root-n HT estimators) Positive $ l_l, l_u$ exist such that, for all $n$, $ l_l \leq n c' \V(\widehat{\delta}^{\tht}_\mathbb{z}) c \leq l_u$ where $\mathbb{z}=\left[y \hspace{2mm} \xx \right]$ and $\left| c \right|=1$ 
	\item[2.] (Convergence of $\widehat{b}$) $\widehat{b}-b=O(n^{-r})$ for some $r>0$
	\item[3.] (Multivariate normal HT estimators) $\left[ \V(\widehat{\delta}^{\tht}_\mathbb{z}) \right]^{-0.5} (\widehat{\delta}^{\tht}_\mathbb{z}-\delta_\mathbb{z})'  \xrightarrow{d} N\left(0, \I \right)$ where $\mathbb{z}=\left[y \hspace{2mm} \xx \right]$
\end{enumerate}

\begin{theorem} Under Assumptions 1-2, $ \sqrt{n}(\widehat{\delta}^{\tr} -\delta ) $ has limiting variance
\begin{align}\label{a.var}
	 \lim_{n\rightarrow \infty} n^{-1} u' \dmat u
\end{align}
where $u := y - \xx b$. Moreover, with the addition of Assumption 3,
\begin{align}\label{a.norm}
n( u' \dmat u)^{-0.5}( \widehat{\delta}^{\tr}-\delta ) \xrightarrow{d} & N( 0, 1 ).
\end{align}
\end{theorem}

\begin{proof} Starting with the form for the generalized regression estimator given in (\ref{greg.b}) and using Assumptions 1 and 2 and the fact that $\E[\widehat{\delta}_{\xx}^\tht]=0$ for all $n$,
\begin{align*}
\widehat{\delta}^{\tr} := &\widehat{\delta}^\tht - \widehat{\delta}_\xx^\tht \widehat{b} 
\\   = &\widehat{\delta}^\tht - \widehat{\delta}_\xx^\tht {b} - \widehat{\delta}_\xx^\tht  (\widehat{b}-b ) 
\\ = &\widehat{\delta}^\tht - \widehat{\delta}_\xx^\tht {b} + O(n^{-0.5-r}) 
\end{align*}
Moreover, $b$ is a fixed (limit) value so that, by Lemma \ref{lemma.finiteVar}, $\V\left( \widehat{\delta}^\tht - \widehat{\delta}_\xx^\tht {b}\right)=n^{-2} u' \dmat u$ for all $n$. Expression (\ref{a.var}) follows. Next, it follows from Assumption 3 that $ \sqrt{n}(\widehat{\delta}^\tht - \widehat{\delta}_\xx^\tht {b}) $ has limiting normal distribution so that, by also invoking the variance expression in (\ref{a.var}), (\ref{a.norm}) follows.
\end{proof}

\section{Optimal Regression For Arbitrary Designs}\label{section.Optimal}

Lemma \ref{lemma.finiteVar} gives the finite sample variance of $\widehat{\delta}^{\tr,f}$, the conjugate of the fixed regression coefficient, $b^f$, first introduced in Section \ref{section.fixedb}. One might next ask, what value of $b^f \in \mathds{R}^{l}$ minimizes the finite sample variance of $\widehat{\delta}^{\tr,f}$? That question is answered in subsection \ref{section.finite.optimal}.  The answer allows the derivation of coefficient estimators that target optimal values of $b^f$ in subsections \ref{section.3HT} and \ref{section.2RA}. By arguments in section \ref{section.asymptotic.arguments}, as long as the proposed coefficient estimators converges to the finite sample optimal value and Assumption 1 holds, then its conjugate ATE estimator obtains the asymptotic minimum variance in the class of generalized regression estimators. 

\subsection{Optimality when $\widehat{b}$ is fixed}\label{section.finite.optimal}

In this section, finite-sample optimal values of $b^f$, the fixed-coefficient introduced in Section \ref{section.fixedb}, are derived. 

\begin{theorem}\label{thrm.bopt}
	Letting $(.)^{(-)}$ represent the Moore-Penrose generalized inverse\footnote{A generalized inverse of $\mathbf{a}$, $\mathbf{a}^{(g)}$, has the property that $\mathbf{a}\mathbf{a}^{(g)}\mathbf{a}=\mathbf{a}$. When the inverse of $\mathbf{a}$ exists, a generalized inverse corresponds to the usual inverse.}, a coefficient value that is finite sample optimal for the fixed-coefficent generalized regression estimator is
	\begin{align}\label{bopt}
	b^{opt}:=(\xx ' \dmat \xx )^{(-)} \xx'\dmat y.
	\end{align}
\end{theorem}
\begin{proof}
	Starting with the finite sample variance of $\widehat{\delta}^{\tr,f}$ given in Theorem \ref{lemma.finiteVar}, we have 
	\begin{align*}
	n^{2} \V\left(\widehat{\delta}^{\tr,f}\right) = & u'\dmat u 
	\\ = & (y-\xx b^f)'\dmat (y-\xx b^f)
	\\ = & y' \dmat y - 2 y' \dmat \xx b^f  + {b^f}' \xx' \dmat \xx b^f
	\end{align*}
	To minimize, take the derivative with respect to $b^f$, set equal to zero, and then rearrange to obtain
	\begin{align}\label{foc}
	(\xx  ' \dmat \xx  )b^f= \xx'  \dmat y.
	\end{align}
	Premultiplying the equality by $(\xx  ' \dmat \xx  )(\xx  ' \dmat \xx  )^{(-)}$ we have
	\begin{align*}
	(\xx  ' \dmat \xx  )(\xx  ' \dmat \xx  )^{(-)}(\xx  ' \dmat \xx  )b^f=& (\xx  ' \dmat \xx  )(\xx  ' \dmat \xx  )^{(-)}\xx'  \dmat y
	\\ \implies (\xx  ' \dmat \xx  )b^f=& (\xx  ' \dmat \xx  )(\xx  ' \dmat \xx  )^{(-)}\xx'  \dmat y
	\end{align*}
	where the second line follows from the definition of a generalized inverse. This implies that
	\begin{align*}
	b^f = & (\xx  ' \dmat \xx  )^{(-)}\xx'  \dmat y
	\end{align*}
	is a solution.
\end{proof}

Defining in terms of a generalized inverse is not simply to account for a few rare cases where the usual inverse is not applicable. There are an infinite number of optimal $b^f$ in common settings, for example, any equal-$\pi_{1}$ design (such as complete randomization) with specification II. 

The choice of the Moore-Penrose generalized inverse, in particular, is arbitrary in a statistical sense. On the one hand, in the special case where $(\xx'\dmat\xx)$ is invertible, all generalized inverses produce the true inverse; in that case, there is a unique $b^f$ vector that minimizes the variance of $\widehat{\delta}^{\tr,f}$.  On the other hand, when $(\xx'\dmat\xx)$ is not invertible, different generalized inverses will lead to different coefficients, all of which are optimal in the sense of minimizing the variance of their respective conjugate ATE estimators. There are two key features recommending the Moore-Penrose generalized, however. First, it has the virtue of being commonly implemented in software. Second, in addition to the generalized inverse property ($\mathbf{a} \mathbf{a}^{(-)} \mathbf{a}= \mathbf{a}$), it has the reflexive property ($\mathbf{a}^{(-)} \mathbf{a} \mathbf{a}^{(-)}  = \mathbf{a}^{(-)} $) which is useful below. 


It may be helpful to discuss briefly what matrices like $\xx' \dmat \xx$ represent.  Just as $n^{-2}y' \dmat y$ gives the variance of HT estimator, so too is $n^{-2} \xx' \dmat \xx$ a variance-covariance matrix of HT estimators. An insight is that the optimal coefficient values are determined by the joint distribution of {\it estimated means} of $x$'s and $y$'s, rather than the joint distribution of $x$'s and $y$'s. This is a slightly different way of thinking about the job of regression adjustment compared to the intuition that one should attempt to approximate the conditional expectation of $y_{1i}$ (or $y_{0i}$) given $x_i$. Instead, one should be more concerned with the conditional expectation of $\widehat{\delta}^\tht$ given $\widehat{\delta}^\tht_\xx$. The former conditional expectation may be well estimated by the latest in machine learning techniques, but, depending on the design, it need not correspond to the latter.



\subsection{A Horvitz-Thompson estimator of $b^{opt}$, namely, {\hththt}
	\\ (-or- Horvitz-Thompson, Horvitz-Thompson, Horvitz-Thompson, trifecta)}\label{section.3HT}

In this section, a HT estimator of $b^{opt}$, given in equation (\ref{bopt}), is introduced.  It has the usual limitations of HT estimator, imprecision and a general lack of invariance to location shifts in $y$. However, the estimator serves as a conceptual starting point, and the refinement in the next subsection may prove more useful.  The coefficient estimator, call it {\hththt}, takes its name from the fact that its conjugate ATE estimator is a constellation of three HT estimators, as can be seen by examining form (\ref{greg.b}).

\begin{definition}[The {\hththt} optimal coefficient estimator]\label{def.3ht}
The ``{\hththt} optimal coefficient estimator" is
\emph{\begin{align}\label{optest}
\widehat{b}^{\thththt}:=&\left( \mathbb{x}' \dmat \mathbb{x} \right)^{(-)} \mathbb{x}' \dmat   \dipi \diR  y
\end{align}}where $(.)^{(-)}$ is the Moore-Penrose generalized inverse.
\end{definition}

\begin{remark}
Note that the estimator differs from a GLS-type estimator in a number of ways. First, the ``denominator" matrix $\left(\xx' \dmat \xx \right)$ is not random. Likewise, the ``numerator'' $\xx' \dmat \bpi^{-1} \R   y$ utilizes the fact that $\xx$ is completely observed. Moreover, with GLS the linear model is assumed and the analogue of the $\dmat$ matrix is designed to minimize the variance of the coefficient vector, which is consistent under the linear model. In the current framework, there is no linear model implied, there are no stochastic errors since potential outcomes are fixed and the $\dmat$ matrix serves to allow the construction of variance-covariance matrices for HT estimators. Precision of the coefficient itself is not guaranteed. Precision guarantees are asymptotic for the conjugate, \emph{$\widehat{\delta}^{\tr,\thththt}$}. 
\end{remark}


\begin{remark}
	Again, the use of the Moore-Penrose generalized inverse is for convenience.  Fortunately, regardless of the generalized inverse chosen in the construction of \emph{$\widehat{b}^{\thththt}$}, the conjugate estimators of the ATE are algebraically equivalent.
\end{remark}

\begin{remark}
	Like any HT estimator, it is unbiased.  To see this, simply take the expectation of (\ref{optest}) and recall that $\E[\R]=\bpi$. 
\end{remark}

\begin{lemma}\label{bopt.is.HT}
	The estimator of ${b}^{opt}$ defined in $(\ref{optest})$ is just an HT estimator of the column sums of the $ 2n \times k $ matrix
	\begin{align*}
	\mathbf{b} &:= \left(\left( \mathbb{x}' \dmat \mathbb{x} \right)^{(-)} \mathbb{x}' \dmat \emph{diag} (y)\right)'.
	\end{align*}
\end{lemma}
\begin{proof}
	The proof involves the recognition that 
	\begin{align*}
	\diR  \dipi y = &  \text{diag}(y) \diR  \dipi 1_{\scriptscriptstyle 2n}.
	\end{align*}
\end{proof}

\subsection{A generalized regression estimator of ${b^{opt}}$, namely, {\rara} 
	  \\(-or- Regression adjusted regression adjustment)}\label{section.2RA}

Recognizing $\widehat{b}^{\thththt}$ in equation (\ref{bopt}) as a HT estimator of the column sums of $\mathbf{b}$ in Lemma \ref{bopt.is.HT}, suggests that an improved estimation strategy may be to recursively apply generalized regression adjustment. No new principles are required. 

The regression adjusted regression coefficient will be called $\widehat{b}^{\trara}$. It takes its name from the fact that its conjugate, $\widehat{\delta}^{\tr,\trara}$, involves two levels of regression adjustment.

Subsequent to its definition, the invariance of its conjugate, $\widehat{\delta}^{\tr,\trara}$, will be proven. Its invariance is notable because its constituent parts are not themselves invariant. 

\begin{definition}[The {\rara} optimal coefficient estimator]\label{TSopt.definition}
The ``{\rara} optimal coefficient estimator" is given by 
\emph{
	\begin{align}\label{TSopt.equation}
\widehat{b}^{\trara}
 :=& \left(\xx ' \dmat \xx \right)^{(-)} \xx' \dmat \left( \bpi^{-1} \R y - \bpi^{-1} \R \xx\widehat{b}^{\pi wls} - \xx\widehat{b}^{\pi wls}\right)
 \\ =&\widehat{b}^{\thththt} -\left(\xx ' \dmat \xx \right)^{(-)}\left( \xx' \dmat \bpi^{-1} \R \xx -  \xx' \dmat \xx \right)\widehat{b}^{\pi wls} \nonumber.
\end{align}
}
where $\widehat{b}^{\pi wls}:=(\xx' \bpi^{-1}\R \xx)\xx' \bpi^{-1}\R y$ is WLS with $\bpi^{-1}$ weights.
\end{definition}
The first line of (\ref{TSopt.equation}) can be compared to (\ref{greg.a}) to make clear that this is regression adjusted regression adjustment. The second line will be at the crux of asymptotic arguments: as long as $\widehat{b}^{\thththt} \xrightarrow{p} {b}^{opt}$, $\widehat{b}^{\pi wls}\xrightarrow{p} b^{\pi wls}$ and $\xx \dmat \bpi^{-1} \R \xx - \xx \dmat \xx\xrightarrow{p} 0$ then $\widehat{b}^{\trara} \xrightarrow{p} {b}^{opt} $. 

Next, the invariance of the two-step optimal regression estimator will be demonstrated, with the help of the following two lemmas.

\begin{lemma}\label{binvariance}
	Let $y^*= e+fy$ where
	\begin{align*}
	e = & 
	c\left[ \begin{matrix}
	-1_{\scriptscriptstyle n} \\ 1_{\scriptscriptstyle n}
	\end{matrix}\right] 
	\end{align*}
	and $c$ and $f$ are arbitrary constants, then for any specification with a constant (e.g., specification $I$) or separate constants for treatment arms (e.g., specification II) the two-step optimal coefficient estimated using $y^*$ instead of $y$ is 
	\emph{\begin{align*}
	\widehat{b}^{\trara}{}^*=& f \widehat{b}^{\trara}+  c \left(\xx \dmat \xx \right)^{(-)}\xx \dmat 
	\left[ \begin{matrix}
	-1_{\scriptscriptstyle n} \\ 1_{\scriptscriptstyle n}
	\end{matrix}\right].
	\end{align*}}
\end{lemma}
\begin{proofatend}
	From the first line of (\ref{TSopt.equation}), the two-step optimal coefficient when inputting $y^*$ in place of $y$ can be written 
	\begin{align*}
	\widehat{b}^{\trara}{}^* =& \left(\xx ' \dmat \xx \right)^{(-)} \xx \dmat \left( \bpi^{-1} \R y^* - \left( \bpi^{-1} \R  - \I \right)\xx\widehat{b}^{\pi wls}{}^*\right)
	\\  =& \left(\xx ' \dmat \xx \right)^{(-)} \xx \dmat \left( \bpi^{-1} \R  - \left( \bpi^{-1} \R  - \I \right)\xx\left( \xx' \bpi^{-1} \R \xx \right)^{-1}\xx' \bpi^{-1} \R\right)\left(fy + c \left[ \begin{matrix}
	-1_{\scriptscriptstyle n} \\ 1_{\scriptscriptstyle n}
	\end{matrix}\right]\right)
	\\  =& f \widehat{b}^{\trara}+ c \left(\xx ' \dmat \xx \right)^{(-)} \xx \dmat \left( \bpi^{-1} \R  - \left( \bpi^{-1} \R  - \I \right)\xx\left( \xx' \bpi^{-1} \R \xx \right)^{-1} \xx' \bpi^{-1} \R\right)\left[ \begin{matrix}
	-1_{\scriptscriptstyle n} \\ 1_{\scriptscriptstyle n}
	\end{matrix}\right]	
	\\  =& f \widehat{b}^{\trara}+ c \left(\xx ' \dmat \xx \right)^{(-)} \xx \dmat \left( \bpi^{-1} \R  - \left( \bpi^{-1} \R  - \I \right)\xx\left( \xx' \bpi^{-1} \R \xx \right)^{-1} \xx' \bpi^{-1} \R\right)\left[ \begin{matrix}
	-1_{\scriptscriptstyle n} \\ 1_{\scriptscriptstyle n}
	\end{matrix}\right]	
	\\  =& f \widehat{b}^{\trara}+ c \left(\xx ' \dmat \xx \right)^{(-)} \xx \dmat \left( \bpi^{-1} \R\left[ \begin{matrix}
	-1_{\scriptscriptstyle n} \\ 1_{\scriptscriptstyle n}
	\end{matrix}\right]	  - \left( \bpi^{-1} \R  - \I \right)\left[ \begin{matrix}
	-1_{\scriptscriptstyle n} \\ 1_{\scriptscriptstyle n}
	\end{matrix}\right]	\right)
		\\  =& f \widehat{b}^{\trara}+ c \left(\xx ' \dmat \xx \right)^{(-)} \xx \dmat \left[ \begin{matrix}
		-1_{\scriptscriptstyle n} \\ 1_{\scriptscriptstyle n}
		\end{matrix}\right]
	\end{align*}
\end{proofatend}
\begin{proof}
	Provided in Appendix.
\end{proof}

\begin{lemma}\label{lem.cons.y}
	Let $y_1=y_0=1_{\scriptscriptstyle n}$, then the finite-sample optimal coefficient is $b^{opt}=\left(\xx \dmat \xx \right)^{(-)}\xx' \dmat \left[ \begin{matrix}
	-1_{\scriptscriptstyle n} \\ 1_{\scriptscriptstyle n}
	\end{matrix}\right]$ and the conjugate of this fixed value has expectation zero and variance zero.
\end{lemma}

\begin{theorem}
 The {\rara} estimator of the ATE, \emph{$\widehat{\delta}^{\tr,\trara}$}, is invariant to scale changes in $y$.
\end{theorem}

\begin{proof}
	As above, let $y^*= e+fy$ where
	\begin{align*}
	e = & 
	c\left[ \begin{matrix}
	-1_{\scriptscriptstyle n} \\ 1_{\scriptscriptstyle n}
	\end{matrix}\right] 
	\end{align*}
	and $c$ and $f$ are arbitrary constants then
	\begin{align*}
	\widehat{\delta}^{\tr ,\trara}{}^* =&  n^{-1} 1'_{\scriptscriptstyle 2n}\bpi^{-1} \R (y^*-\xx \widehat{b}^{\trara}{}^*  ) +n^{-1} 1'_{\scriptscriptstyle 2n} \xx \widehat{b}^{\trara}{}^* \nonumber
	\\ =&  n^{-1} 1'_{\scriptscriptstyle 2n}\bpi^{-1} \R \left(f(y-\xx \widehat{b}^{\trara})+ e-\xx \left(\xx' \dmat \xx \right)^{(-)}\xx ' \dmat e \right) +n^{-1} 1'_{\scriptscriptstyle 2n} \left(f\xx \widehat{b}^{\trara}+\xx \left(\xx' \dmat \xx \right)^{(-)}\xx ' \dmat e\right) 
	\\ =&  f \widehat{\delta}^{\tr,\trara}+n^{-1} 1'_{\scriptscriptstyle 2n}\bpi^{-1} \R \left(e-\xx \left(\xx' \dmat \xx \right)^{(-)}\xx ' \dmat e \right) +n^{-1} 1'_{\scriptscriptstyle 2n} \left(\xx \left(\xx' \dmat \xx \right)^{(-)}\xx ' \dmat e\right)
	\\ =&  f \widehat{\delta}^{\tr,\trara}.
	\end{align*}
The last line follows from Lemma \ref{lem.cons.y}
\end{proof}

\begin{theorem}
	 The {\rara} estimator of the ATE, \emph{$\widehat{\delta}^{\tr,\trara}$}, is invariant to scale changes in $\xx$.  
\end{theorem}
\begin{proof}
	Let $\f$ be a $(l \times l)$ transformation matrix such that $\f^{-1}$ exists and let $\xx^{*}=\xx \f$.  Next, write the two-step optimal estimator of the ATE computed with $\xx^{*}$ in place of $\xx$ as
	\begin{align*}
	\widehat{\delta}^{\tr, \trara}{}^*= \widehat{\delta}^{\tht}-n^{-1}1_{\scriptscriptstyle 2n}\left(\bpi^{-1} \R - \I \right)\xx^* \left({\xx^*}'\dmat \xx^* \right)^{(-)}{\xx^*}'\dmat \left( \bpi^{-1} \R y -\left(\bpi^{-1}\R - \I \right) \xx^*\widehat{b}^{\pi wls}{}^* \right)  
	\end{align*}
	and note that $\xx^*{\widehat{b}^{\pi {wls}}}{}^*=\xx\widehat{b}^{\pi wls}$ by the invariance of WLS.  Now note that 
	\begin{align*}
	\xx^* \left({\xx^*}'\dmat \xx^* \right)^{(-)}{\xx^*}' =&  \xx  \f \left(\f' {\xx }'\dmat \xx  \f \right)^{(-)}\f' {\xx }'
	\\ =&  \xx \f \hspace{1mm} \f^{-1} \left( {\xx }'\dmat \xx  \right)^{(-)}\f'^{-1} \f' {\xx }'
	\\ =&  \xx \left( {\xx }'\dmat \xx  \right)^{(-)} {\xx }'
	\end{align*}
	where the second line follows from the properties of generalized inverses \citep{campbellmeyer}. Hence, $\widehat{\delta}^{\tr,\trara}{}^*=\widehat{\delta}^{\tr,\trara}$.
\end{proof}

\begin{remark}
	Given its definition, \emph{${\widehat{\delta}^{\tr,\trara}}:=\widehat{\delta}^\tht -\widehat{\delta}^\tht_\xx \widehat{b}^{\trara}$}, invariance to location shifts in $y$ and $\xx$ not immediately obvious because the constituent parts, \emph{($\widehat{\delta}^\tht, \widehat{\delta}^\tht_\xx$, and $\widehat{b}^{\trara})$}, are not generally invariant. By contrast, the optimal generalized regression estimator, \emph{$\widehat{\delta}^{\tr,\thththt}$}, is only invariant to location shifts in special cases (e.g., complete randomization).
\end{remark}


\subsection{Conclusions about the proposed optimal estimators, $\widehat{\delta}^{\tr,\thththt}$ and $\widehat{\delta}^{\tr,\trara}$}

Estimators $\widehat{\delta}^{\tr,\thththt}$ and $\widehat{\delta}^{\tr,\trara}$ have the virtue of being asymptotically optimal for arbitrary designs. However, asymptotic optimality does not necessarily imply good finite sample performance, and $\widehat{\delta}^{\tr,\thththt}$ is not recommended in practice because it is unnecessarily imprecise and not generally invariant to location shifts in $y$. $\widehat{\delta}^{\tr,\trara}$ may be useful in some cases.

Alternatives to $\widehat{\delta}^{\tr,\thththt}$ and $\widehat{\delta}^{\tr,\trara}$ are available for specific designs. In Section \ref{section.completelyrand}, complete randomization is considered, followed by Section \ref{section.cluster} on clustered randomization. The sections show how to derive optimal estimators specific to those designs from this framework. Some of the results are known, but the derivation helps connect the framework herein to prior work \citep[e.g.][]{lin}.


\section{Variance Estimation \\ (-or, more exactly- Variance Bound Estimation)}\label{section.var.est}

In general, the variance expressions of the form (\ref{HTvar}) and (\ref{eq.finiteVar}) are not identified. This is due to the fact that some pairs of elements in the vector $y$ can never be jointly observed, and hence, for example, some terms in the quadratic $n^{-2} y'\dmat y$ are never observable. One reason is that a given unit's potential outcomes, $y_{0i}$ and $y_{1i}$, can never be observed together. This problem is referred to as the ``fundamental problem of causal inference" \citep{holland}.  But other design features, such as clustering or pair randomization, render various combinations of potential outcomes jointly unobservable as well.

Starting with Neyman (1923) one proposed solution to unidentified variance has been to estimate a {\it variance bound}, i.e., a quantity that is known to be greater than the variance, but which is identified. Conservative inference follows.

In Section \ref{section.boundingvar}, the terms ``variance bound" and ``identified variance bound" are defined in terms of the current framework. Framing the problem in matrix terms facilitates insight and leads to methods of comparing alternative bounds.  In Section \ref{section.ASbound}, an important variance bound that has the virtue of being identified in any identified design, the Aronow-Samii (AS) bound, is defined. The AS bound serves as a benchmark against which other potential bounds might be compared. After that, Section \ref{section.Mbound} proposes an algorithm for finding an alternative variance bound, which can improve upon the AS bound substantially in some cases.  Finally, Section \ref{section.EstimatingBound} addresses the subject of how to estimate a variance bound, both for HT estimators and generalized regression estimators. And finally, in Section \ref{section.BorrowingBound} a tighter variance bound specifically for the {\rara} estimator is proposed that can substantially narrow intervals for that estimator.

\subsection{Bounding the variance}\label{section.boundingvar}

\begin{definition}[Variance bound]
	For an arbitrary $2n \times 2n$ matrix $\tilde{\dmat}$, let $n^{-2} y'\tilde{ \dmat }y$ be a ``bound" for the variance $n^{-2}y'\dmat y$ if, for all $y\in\mathds{R}^{2n}$, $n^{-2}y'\dmat y \leq n^{-2} y'\tilde{\dmat}y$. 
\end{definition}

\begin{definition}[Tighter variance bound]\label{def.tighter}
	For two $2n \times 2n$ matrices, $\tilde{\dmat}^a$ and $\tilde{\dmat}^b$, that correspond to different variance bounds, $\tilde{\dmat}^a$ corresponds to a ``tighter variance bound" if for all $y \in \mathds{R}^{2n}$ $n^{-2}y'\tilde{\dmat}^ay \leq n^{-2} y'\tilde{\dmat}^by$. 
\end{definition}

\begin{definition}[Tighter variance bound under the sharp null]
	First, let $y_0$ and $y_1$ represent length-$n$ vectors of control and treatment potential outcomes, respectively, and recall that under the sharp null $y_1=y_0$. For two matrices $2n \times 2n$ matrices, $\tilde{ \dmat }^a$ and $\tilde{ \dmat }^b$, that correspond to different variance bounds, $\tilde{\dmat}^a$ corresponds to a ``tighter variance bound under the sharp null" if, for all $y_1=y_0 \in \mathds{R}^{n}$, $n^{-2}y'\tilde{\dmat}^a y \leq n^{-2} y'\tilde{\dmat}^b y$. 
\end{definition}

\begin{lemma}\label{psd}
	$n^{-2} y'\tilde{ \dmat }y$ is a bound for the variance $n^{-2}y'\dmat y$ if and only if matrix $\tilde{ \dmat }-\dmat$ is positive semi-definite. 
\end{lemma}

\begin{proof}
	By the definition of a bound, $n^{-2} y'\tilde{\dmat}y - n^{-2}y'\dmat y \geq 0$ for all $y\in \mathds{R}^n$. This implies that $n^{-2} y'(\tilde{\dmat}-\dmat ) y \geq 0$, i.e., that $ \tilde{\dmat}-\dmat  $ is positive semi-definite.  
\end{proof}

\begin{corollary}\label{tighter.bound}
	For two $2n \times 2n$ matrices that define variance bounds, $\tilde{\dmat}^a$ and $\tilde{\dmat}^b$, $\tilde{\dmat}^a$ corresponds to a ``tighter variance bound" if and only if the matrix $ \tilde{\dmat}^b-\tilde{\dmat}^a$ is positive semi-definite.
\end{corollary}

\begin{corollary}\label{tighter.bound.sn}
	First writing the four $n \times n$ partitions of a matrix $\tilde{\dmat}$ as $\tilde{\dmat}_{00}$, $\tilde{\dmat}_{01}$, $\tilde{\dmat}_{10}$, and $\tilde{\dmat}_{11}$, define an $n \times n$ matrix
	\begin{align*}
	\tilde{\dmat}_{\scriptscriptstyle  + }:=\tilde{\dmat}_{00}+\tilde{\dmat}_{11}-\tilde{\dmat}_{10}-\tilde{\dmat}_{01}.
	\end{align*}
	Then for two matrices, $\tilde{ \dmat }^a$ and $\tilde{ \dmat }^b$, that correspond to different variance bounds, $\tilde{\dmat}^a$ corresponds to a ``tighter variance bound under the sharp null" if and only if $\tilde{ \dmat }_{\scriptscriptstyle +}^b-\tilde{ \dmat }_{\scriptscriptstyle +}^a$ is positive semi-definite. 
\end{corollary}

\begin{remark}\label{remark.tighter.bounds}
	Lemma \ref{psd} implies that a test for whether a candidate $\tilde{\dmat}$ corresponds to a bound is to check the eigenvalues of $\tilde{ \dmat }-\dmat$ for nonnegativeness. Similarly, Corollary \ref{tighter.bound} implies that a test for whether a candidate $\tilde{\dmat}^a$ matrix corresponds to a tighter variance bound than $\tilde{\dmat}^b$ is to check the eigenvalues of $\tilde{\dmat}^b-\tilde{\dmat}^a$ for non-negativeness. Failing this test, Corollary \ref{tighter.bound.sn} says that an adjudication between bounds might still be made by testing $\tilde{ \dmat }_{\scriptscriptstyle +}^b-\tilde{ \dmat }_{\scriptscriptstyle +}^a$ for non-negative eigenvalues. Alternatively, if some eigenvalues of $\tilde{\dmat}^b-\tilde{\dmat}^a$ are positive and some negative, heuristic methods of choosing the better bound could include comparing the maximum and minimum eigenvalues or determining the sign of the sum of eigenvalues. Biased inference due to ``cherry picking" the narrower confidence interval can be avoided because such comparisons can be done before observing the outcome variable.
\end{remark}

\begin{definition}[Identified variance bound]\label{def.identified.bound}
	For an arbitrarily defined $2n \times 2n$ matrix $\tilde{\dmat}$, let $n^{-2}y'\tilde{\dmat} y$ be an ``identified variance bound" for $n^{-2}y'\dmat y$ if it is a variance bound and if 
	\begin{align*}
	\bfI (\dmat=-1) 
	\circ \bfI(\tilde{\dmat}= 0)=\bfI (\dmat=-1)
	\end{align*}
	where $\circ$ is element-wise multiplication and, for example, $\bfI\left(\dmat=-1\right)$ is an indicator function returning an $2n \times 2n$ matrix of ones and zeros indicating whether each element of $\dmat$ is equal to $-1$ (an indication that the associated term in the variance quadratic is impossible to observe).	
\end{definition}
The definition says that for a variance bound $n^{-2}y'\tilde{\dmat}y$ to be an identified bound the elements of matrix $\dmat$ equal to $-1$ must correspond to elements of $\tilde{\dmat}$ that equal 0.

\subsection{Defining the Aronow-Samii bound}\label{section.ASbound}

Consider an identified bound proposed by \cite{aronowsamii17} that has the a unusual virtue of being perfectly general, i.e., applicable to arbitrary (identified) designs.

\begin{definition}[Aronow-Samii variance bound]
	The ``Aronow-Samii variance bound" is
	\begin{align}\label{varUB}
	\tilde{ \emph{\V} }^\tas \left(\widehat{\delta}^{\tht} \right) := n^{-2} y' \tilde{\dmat}^\tas y
	\end{align}
	where
	\begin{align*}
	\tilde{\dmat}^\tas:=\dmat +\bfI\left(\dmat=-1\right)+\emph{diag}\left(\bfI\left(\dmat=-1\right) 1_{\scriptscriptstyle 2n}\right)
	\end{align*}
	and $\emph{diag}(.)$ creates a diagonal matrix from a vector.
\end{definition}
\begin{theorem}
	The Aronow-Samii variance bound, $n^{-2}y' \tilde{\dmat}^\tas y$,  is an identified bound for $n^{-2}y' {\dmat} y$.
\end{theorem}
\begin{proof}
	By definition of $\tilde{\dmat}^\tas$, 
	\begin{align*}
	\tilde{\dmat}^\tas-\dmat= \bfI\left(\dmat=-1\right)+\text{diag}\left(\bfI\left(\dmat=-1\right) 1_{\scriptscriptstyle 2n}\right).
	\end{align*}
	Note that $\bfI\left(\dmat=-1\right)+\text{diag}\left(\bfI\left(\dmat=-1\right) 1_{\scriptscriptstyle 2n}\right)$ sets the diagonal elements of ($\tilde{\dmat}^\tas-\dmat$) equal to the sum of off-diagonal elements in row $i$ (which by construction are all non-negative). The Gershgorin circle theorem implies that a real matrix is positive semi-definite if, for all $i$, diagonal element $ii$ is greater or equal to the sum of the absolute values of the other elements in the $i^{th}$ row. So, by the Gershgorin circle theorem $\tilde{\dmat}^\tas-\dmat$ is positive semidefinite. Therefore, by Lemma (\ref{psd}), $n^{-2} y'\tilde{\dmat}^\tas y$ is a variance bound. Moreover, as long as the design is an identified design (i.e., $0<\pi_{1i}<1$ for all $i$), it is an identified bound because $\bfI\left(\dmat=-1\right)$ ensures that the elements of $\dmat$ equal to $-1$ correspond to 0's in $\tilde{\dmat}^{\tas}$.
\end{proof}

Aronow and Samii (2017) derive the bound using Young's inequality.\footnote{In sum, since, for example, $ -y_{0i}$ and $y_{1i} $ are impossible to jointly observe for all $i$ (since units can only be assigned to treatment or control), the unobservable quantity $2y_{1i}y_{0i}$ which appears in the quadratic in (\ref{HTvar}) is bounded by the addition of identified quantity $y_{1i}^2+y_{0i}^2$ in equation (\ref{varUB}). By Young's inequality $2y_{1i}y_{0i} \leq y_{1i}^2+y_{0i}^2$, and hence ${\V}\left(\widehat{\delta}^{\tht} \right) \leq \tilde{\V}^\tas\left(\widehat{\delta}^{\tht} \right)$. } The above-theorem and proof using the Gershgorin circle theorem tie their insight to the current framework.  

The AS variance bound is elegant in its universal applicability and serves as a benchmark against which proposed alternatives might be compared. In some cases it can behave quite reasonably. For example, in completely randomized experiments it is exact under the sharp null. That said, simulation examples suggest it can be dramatically over-conservative for some designs.

In the next subsection an algorithm is proposed that may obtain a tighter identified bound for arbitrary designs. Additionally, an analytically-defined bound for cluster-randomized experiments is proposed in Section \ref{section.cluster} which is exact under the sharp null and provably tighter than the AS bound. 

\subsection{A proposed algorithm for a tighter variance bound}\label{section.Mbound}

The following is an algorithm which, if it converges, obtains an identified variance bound. Like the AS bound it has the virtue of being widely applicable. The drawback is the potential computationally difficulty.

In some cases the proposed bound is strictly tighter than the AS bound in the sense of Definition \ref{tighter.bound}. However, even when not strictly tighter, it can be tighter under the sharp null. In practice, relative tightness can be evaluated on a case-by-case basis using Corollary \ref{tighter.bound} and Corollary \ref{tighter.bound.sn} and Remark \ref{remark.tighter.bounds}.

\begin{algorithm} 
\end{algorithm}
		\begin{enumerate}
			\item Set $\mathbf{t}=\bfI(\dmat=-1)$
			\item Obtain the eigen decomposition of matrix $\mathbf{t}$
			\item Update $\mathbf{t}=\mathbf{v} (\mathbf{e} \circ \bfI(\mathbf{e}>0)) \mathbf{v}'$ where $\mathbf{v}$ is the matrix of eigenvectors and $\mathbf{e}$ is a diagonal matrix of eigenvalues
			\item Update $\mathbf{t}=\bfI(\dmat=-1)+\bfI(\dmat\neq -1) \circ \mathbf{t}$
			\item Repeat steps 2-4 until convergence is achieved (i.e., until all eigenvalues are non-negative in step 2) 
			\item Set $\tilde{ \dmat }^\tm =\dmat+\mathbf{t}$
		\end{enumerate}	
As above, $\circ$ is elementwise multiplication and, for example, $\bfI(\mathbf{e}>0)$ is an indicator function returning a matrix of ones and zeros indicating which elements of $\mathbf{e}$ are greater than zero.
 		
Conceptually, the goal of the algorithm is to create a matrix $\mathbf{t}$ that can be added to $\dmat$ yielding a $\tilde{ \dmat }$ matrix that corresponds to an identified variance bound. By Lemma \ref{psd} and Definition \ref{def.identified.bound}, there are two requirements for $\mathbf{t}$. First it must be positive semi-definite, and, second, elements corresponding to $-1$'s in the matrix $\dmat$ must equal one. In step 1, $\mathbf{t}$ meets the second criterion, but not the first.  In step 3, the algorithm creates an approximation to the matrix $\bfI(\dmat=-1)$ by way of the eigen decomposition that ensures positive semi-definiteness, thus meeting the first criterion. However, due to the approximation, $\mathbf{t}$ no longer meets the second criterion. Therefore, in step 4 the algorithm forces $\mathbf{t}$ to have 1's wherever $\dmat$ has $-1$'s in order to again meet the second criteria. But doing so means that $\mathbf{t}$ will no longer meet the first criteria. So, the algorithm iterates through steps 2-4 until convergence is achieved (i.e., until all eigenvalues are non-negative in step 2) at which point $\mathbf{t}$ meets both criteria and, thus, $\tilde{ \dmat }^\tm$ corresponds to an identified bound.

\subsection{Estimating a variance bound}\label{section.EstimatingBound}

For a $\tilde{\dmat}$ associated with an identified variance bound for an HT estimator, $\tilde{\V}\left(\widehat{\delta}^{\tht} \right)$, an unbiased estimator of that bound can then be constructed as
\begin{align}\label{ub_var_est}
\widehat{\tilde{\V}}\left(\widehat{\delta}^{\tht} \right) := n^{-2} y' \R\left( \tilde{\dmat} / \tilde{\matp}\right) \R y
\end{align}
where $/$ is element-wise division, 
\begin{align}\label{ptilde}
\tilde{\matp} := \left[\begin{array}{ccc}
\matp_{00} & \matp_{01} 
\\ \matp_{10} &\matp_{11}\end{array}\right]
+ \left[\begin{array}{ccc}
\bfI\left( \matp_{00}=0 \right) & \bfI\left( \matp_{01}=0 \right)
\\ \bfI\left( \matp_{10}=0 \right) &\bfI\left( \matp_{11}=0 \right) \end{array}\right],
\end{align}
$\matp_{00}$ has $ij$ element $\pi_{0i0j}$, and $\matp_{01}$ has $ij$ element $\pi_{0i1j}$. $\matp_{10}$ and $\matp_{11}$ are defined analogously. Conceptually, $\tilde{\matp}$ weights each term in the quadratic in (\ref{ub_var_est}) inversely proportional to the probability of observing that term. The second term on the right hand side of equation (\ref{ptilde}) serves to replace zeros in $\matp_{00}$, $\matp_{01}$, $\matp_{10}$ and $\matp_{11}$ with ones to ensure that there is no division by zero in $\tilde{\dmat} / \tilde{\matp}$. The choice of replacing the zeros with a value of one is arbitrary and does not affect the estimate because zero elements in the first term on the right-hand-side of (\ref{ptilde}) correspond to zeros in $\tilde{\dmat}$ (by the definition of an identified variance bound). 
Unbiasedness follows from the recognition that $\E \left[ \R \left(\tilde{\dmat}/\tilde{\matp} \right) \R \right]=\tilde{\dmat}$.


Next, by analogy to equation (\ref{ub_var_est}) and an appeal to the asymptotics, one can motivate the variance bound estimator for generalized regression estimators as
\begin{align}\label{avarboundest}
\widehat{\tilde{\V}}\left(\widehat{\delta}^\tr \right) := & n^{-2} \widehat{u}' \R \left( \tilde{\dmat}/ \tilde{\matp} \right) \R  \widehat{u}
\end{align}
where $\widehat{u}=y-\xx \widehat{b} $.  

Unless $\widehat{b}$ is a fixed constant vector, its introduction in (\ref{avarboundest}) means that the bound estimator is not generally unbiased. Refinements made to White's HC0 variance estimator could be applied here, such as degrees of freedom adjustments.

\subsection{Borrowing a tighter variance bound for $\widehat{\delta}^{\tr,\trara }$}\label{section.BorrowingBound}

In many instances the gap between the variance of the $\widehat{\delta}^{\tr,\trara}$, defined in section \ref{section.2RA}, and the bound on its variance estimated by (\ref{avarboundest}) is exceedingly large, leading to overly conservative inference. Theorem \ref{UBest} introduces an approach to minimizing the variance {\it bound} of the fixed-coefficient generalized regression estimator, $\widehat{\delta}^{\tr,f}$, with respect to $b^f$. Then Theorem \ref{UBborrow} gives a justification for ``borrowing" its variance bound estimator and pairing it with $\widehat{\delta}^{\tr,\trara}$ for the purpose of inference.  

\begin{theorem}\label{UBest}
	Let $\tilde{\dmat}$ correspond to a variance bound and let $u=y-\xx b^f$, where $b^f$ is a fixed constant vector.  Then for the fixed-coefficient generalized regression estimator, a value of $b^f$ that minimizes variance bound $n^{-2}u'\tilde{ \dmat } u$ is 
	\begin{align}\label{bopt.tilde}
	\tilde{b}^{opt}:=(\xx' \tilde{\dmat} \xx )^{(-)} \xx' \tilde{\dmat} y.
	\end{align}
\end{theorem}
\begin{proof}
	The result follows the same logic as the optimal finite sample ${b^f}$ in Theorem \ref{thrm.bopt}.  Rather than minimizing ${u}' {\dmat} {u}$, however, simply minimize ${u}' \tilde{\dmat} {u}$ with respect to $b^f$ where ${u}=y-\xx b^f$.
\end{proof}

\begin{theorem}\label{UBborrow}
	Let $\tilde{\dmat}$ correspond to a variance bound and define $\tilde{u}=y-\xx\tilde{b}^{opt}$ and ${u}=y-\xx{b}^{opt}$, then for all $y\in \mathds{R}^{2n}$, $n^{-2} u' \dmat u \leq n^{-2} \tilde{u}'\tilde{\dmat}\tilde{u}\leq n^{-2} {u}'\tilde{\dmat}{u}$. Hence, $n^{-2} \tilde{u}'\tilde{\dmat}\tilde{u}$ is a tighter bound for the variance of the fixed-coefficient generalized regression estimator with optimal coefficient $b^{opt}$ than $n^{-2} {u}'\tilde{\dmat}{u}$. 
\end{theorem}

\begin{proof}
	By Theorem \ref{thrm.bopt}, because $b^{opt}$ minimizes the variance of the fixed-coefficient generalized regression estimator, $n^{-2} u' \dmat u \leq n^{-2} \tilde{u}'{\dmat}\tilde{u}$.  Moreover, because $\tilde{\dmat}$ corresponds to a variance bound, $n^{-2} \tilde{u}'{\dmat}\tilde{u} \leq n^{-2} \tilde{u}'\tilde{\dmat}\tilde{u}$. Finally, by Theorem \ref{UBest}, because $\tilde{b}^{opt}$ minimizes the variance bound of the fixed-coefficient generalized regression estimator, $n^{-2} \tilde{u}'\tilde{\dmat}\tilde{u} \leq n^{-2} {u}'\tilde{\dmat}{u}$. The result follows. 
\end{proof}

The result motivates the variance bound estimator for $\widehat{\delta}^{\tr,\trara }$, 
\begin{align}\label{boptbound}
\widehat{\tilde{\V}}\left(\widehat{\delta}^{\tr,\trara } \right) := & n^{-2} \widehat{\tilde{u}}' \R \left( \tilde{\dmat}/ \tilde{\matp} \right) \R  \widehat{\tilde{u}},
\end{align}
where $\widehat{\tilde{u}}=y-\xx \widehat{\tilde{b}}{}^{opt}$ and $\widehat{\tilde{b}}{}^{opt}$ is an estimator of (\ref{bopt.tilde}) that, for example, could be defined using a regression adjustment procedure analogous to (\ref{TSopt.equation}). Again, additional adjustments for degrees of freedom or leverage may be advisable.


\section{Optimal Regression for Complete Randomization}\label{section.completelyrand}

This section draws the connection to two optimal estimators for completely randomized designs, Lin's OLS for specification II and the tyranny of the minority estimator. That these estimators are asymptotically optimal is known, but the proofs are novel and the demonstration connects the current framework to those results. In this section it is also shown that Lin's fully-interacted specification leads to tighter bounds on the variance in (\ref{varUB}). Finally, it is proven that, for specification II, the {\rara} estimator, $\widehat{\delta}_\II^{\tr,\trara}$, is algebraically equivalent to the OLS estimator, $\widehat{\delta}_\II^{\tr,ols}$.  In that sense, {\rara} can be thought of as a generalization of OLS with specification II for arbitrary designs.

\subsection{OLS is optimal for completely randomized designs with specification II}\label{section.ols.opt}

In this subsection, it will be shown that the population OLS coefficient, call it $b^{ols}_\II$, is an optimal coefficient for the fixed-coefficient generalized regression estimator with completely randomized designs and specification II. It will follow that, under Assumptions 1 and 2 in Section \ref{section.asymptotic.arguments}, the OLS coefficient {\it estimator}, call it $\widehat{b}^{ols}_\II$, has a conjugate ATE estimator, $\widehat{\delta}^{\tr,ols}_\II$, that obtains minimum asymptotic variance.

\begin{definition}[OLS coefficient]\label{def.ols}
	The ``OLS coefficient for specification \emph{II}" is 
	\emph{\begin{align*}
		b^{ols}_\II= &\left({\xx_\II}' \xx_\II  \right)^{-1} {\xx_\II }' y 
		\\ = &\left[\begin{matrix} 
		\mu_{y_0} \\ \var(\x)^{-1}\cov(\x, y_0)
		\\ \mu_{y_1} \\ \var(\x)^{-1}\cov(\x, y_1)
		\end{matrix}\right],
		\end{align*}}where $\var(\tx):=n^{-1}\sum_i(\tx_i-\mu_{\tx})(\tx_i-\mu_{\tx})'$ and $\cov(\tx, y_1):=n^{-1}\sum_i(\tx_i-\mu_{\tx})(y_{1i}-\mu_{y_1})'$ are finite population variance and covariance, respectively.\footnote{Note that $\var(\tx)$ and $\cov(\tx, y_1)$ summarize features of the finite population. They should not be taken to imply randomness in $\x$ and $y_1$. By contrast, $\V(.)$ is used throughout to characterize the design variance of an estimator (or variance-covariance of a vector of estimators, depending on context).}
\end{definition}
\begin{definition}[OLS coefficient estimator]\label{def.ols.est} The ``OLS coefficient estimator for specification \emph{II}" is
	\emph{\begin{align*}
		\widehat{b}^{ols}_\II=& \left({\xx_\II}' \R \xx_\II \right)^{-1} {\xx_\II}' \R y.
		\end{align*}}
\end{definition}

To show that in completely randomized experiments with specification II the coefficient in Definition \ref{def.ols} is optimal for the fixed-coefficient generalized regression estimator, the entire set of optimal coefficients for an arbitrary design is first defined. Subsequently, it can be shown that, for complete randomization, $b^{ols}_\II$ is in that set. The infinite set of optimal coefficients for an arbitrary design is given in the following Lemma.

\begin{lemma}\label{lemma.allbopt}
	First, for an arbitrary design, for any given generalized inverse, denoted $(.)^{(g)}$, and a given $z \in \mathds{R}^{\colsxx}$ where $l$ is the number of columns of $\xx$, let
	\begin{align}\label{boptz}
	b^{opt,gz}:=\left(\xx ' \dmat \xx \right)^{(g)}\xx'\dmat y +\left(\I_\colsxx-(\xx ' \dmat \xx)^{(g)}(\xx ' \dmat \xx) \right)z
	\end{align}
	where $\I_\colsxx$ is an $\colsxx \times \colsxx$ identity matrix.  Then the entire set of solutions to $(\ref{foc})$ can be defined as 
	\begin{align}\label{allGinv}
	\{ b^{opt,gz}\left. \right| z \in \mathds{R}^{\colsxx} \}.
	\end{align}
\end{lemma}
\begin{proofatend}
 
	The proof consists of two parts: first proving that all members of the set in $(\ref{allGinv})$ are solutions and, second, showing that all solutions are in the set.  
	
	First note that the fact that (\ref{bopt}) is a solution to (\ref{foc}) implies that $(\xx ' \dmat \xx )(\xx ' \dmat \xx )^{(-)} \xx'\dmat y=\xx'\dmat y$.  Next, premultiplying (\ref{boptz}) by $(\xx' \dmat \xx)$ yields
	\begin{align*}
	(\xx' \dmat \xx)b^{opt,z}=&(\xx '\dmat \xx)\left(\xx ' \dmat \xx \right)^{(-)}\xx'\dmat y +(\xx \dmat \xx)\left(\I_l-(\xx ' \dmat \xx)^{(-)}(\xx ' \dmat \xx) \right)z
	\\ \implies (\xx' \dmat \xx)b^{opt,z}=&\xx'\dmat y +\left((\xx' \dmat \xx)-(\xx' \dmat \xx)(\xx ' \dmat \xx)^{(-)}(\xx ' \dmat \xx) \right)z
	\\ \implies (\xx' \dmat \xx)b^{opt,z}=&\xx'\dmat y
	\end{align*}
	Hence, $b^{opt,z}$ is a solution to $(\ref{foc})$.  This proves that all members of the set given by (\ref{allGinv}) are solutions.  
	
	Next, to prove that all solutions are in the set given by (\ref{allGinv}), let $b^{opt,*}$ represent a an arbitrary solution to $(\ref{foc})$ and then set $z=b^{opt,*}$. Then  
	\begin{align*}
	b^{opt,z}=&\left(\xx ' \dmat \xx \right)^{(-)}\xx'\dmat y +\left(\I_l-(\xx ' \dmat \xx)^{(-)}(\xx ' \dmat \xx) \right)b^{opt,*}
	\\ =&\left(\xx ' \dmat \xx \right)^{(-)}\xx'\dmat y +\left(b^{opt,*}-(\xx ' \dmat \xx)^{(-)}(\xx ' \dmat \xx)b^{opt,*} \right).
	\\ =&\left(\xx ' \dmat \xx \right)^{(-)}\xx'\dmat y +\left(b^{opt,*}-(\xx ' \dmat \xx)^{(-)}\xx'\dmat y\right)
	\\ =& b^{opt,*}	.
	\end{align*}
	Hence, all solutions are represented in the set given by (\ref{allGinv}).
 
\end{proofatend}
\begin{proof}
	Provided in Appendix.  
\end{proof}

In equation (\ref{boptz}), the generalized inverse, $g$, is considered fixed and the set is defined with regard to all possible $z$. That said, $g$ could be any generalized inverse.  The point is that the entire set can be defined with reference to only a single generalized inverse. In keeping with the prior use of the Moore-Penrose generalized inverse above, it might have been sensible to also use it in (\ref{boptz}) instead of the more generic $g$. However, in order to prove that OLS is optimal, a subsequent proof will use the fact that $g$ can be some other inverse.

Next, before it can be proven that $b^{ols}_\II$ is in the set given by Lemma \ref{lemma.allbopt}, it must be shown that a ``separable" solutions can be optimal. By separable, it is meant that the sub-vector of coefficients associated with treatment units does not involve the terms $\dmat_{00}$, $\dmat_{01}$, $\dmat_{10}$ or $y_0$ (the vector of control potential outcomes), and, likewise, the sub-vector of coefficients associated with control potential outcomes does not involve the terms $\dmat_{11}$, $\dmat_{01}$, $\dmat_{10}$, or $y_1$ (the vector of treatment potential outcomes).\footnote{Being separable implies that one way to minimize the overall variance of $\widehat{\delta}$ is to separately minimize (with respect to $b$) the variance of the estimated mean of each experimental arm while ignoring the other arm.} 

Separability is provable with a less restrictive assumption than complete random assignment, namely, under equal-$\pi_1$ designs, (i.e., designs where $\pi_{1i}=\pi_{1j}$ for all $i,j$). Equal-$\pi_1$ designs include complete randomization, Bernoulli designs, cluster-randomized designs (i.e., complete random assignment of clusters) and block randomized designs where an equal fraction is assigned to treatment in every block.

\begin{lemma}\label{LA1} For designs where $\pi_{1i}=\pi_{1j}$ for all $i,j$ (e.g., completely randomized designs), defining $\dmat_{**}$ to be the matrix with $ij$ element $\pi_{1i1j}-\pi_{1i}\pi_{1j}$, the following equalities hold:
	\begin{align*}
	\dmat_{**} =  \pi_{0i}^2   \dmat_{00}   = \pi_{1i}^2   \dmat_{11}   = -\pi_{1i} \pi_{0i} \dmat_{10}   = -\pi_{0i} \pi_{1i}   \dmat_{01}  .
	\end{align*}
\end{lemma}
\begin{proof}
	The result follows from the $\pi_{1i}=\pi_{1j}$ for all $i,j$ and the defintion of the four partitions of $\dmat$ given in (\ref{dmat.quad}).
\end{proof}
\begin{lemma}\label{lemma.sep} 
	Let $\tilde{ \x}=[1_n \hspace{1mm} \x]$ be the matrix of coefficients with the addition of a leading constant. For designs where $\pi_{1i}=\pi_{1j}$ for all $i,j$ (e.g., completely randomized designs), the fixed-coefficient generalized regression estimator with the ``separated" coefficient
	\emph{\begin{align}\label{equation.sep}
		b^{sep}_\II= 
		\left[\begin{matrix}
		(\tx ' \dmat_{00} \tx )^{(-)} \tx'\dmat_{00} y_0 \\ (\tx ' \dmat_{11} \tx )^{(-)} \tx'\dmat_{11} y_1 
		\end{matrix} \right]
		\end{align}}has finite-sample minimum variance, i.e., $b^{sep}_\II$ is in the set of optimal fixed-coefficients given by Lemma \ref{lemma.allbopt}.
\end{lemma}

\begin{proof}
	From \cite{rhode}, for a positive semi-definite symmetric matrix
	\begin{align*}
	\m=\left[\begin{matrix}
	\mathbf{a} & \matc \\ \matc' & \mathbf{b}
	\end{matrix}\right]
	\end{align*}
	a generalized inverse, call it $g$, is given by
	\begin{align}\label{GenInv}
	\mathbf{m}^{(g)}=\left[\begin{matrix}
	\mathbf{a}^{(-)}+\mathbf{a}^{(-)}\matc \mathbf{q}^{(-)}\matc' \mathbf{a}^{(-)} & -\mathbf{a}^{(-)}\matc \mathbf{q}^{(-)} \\ -\mathbf{q}^{(-)}\matc'\mathbf{a}^{(-)} & \mathbf{q}^{(-)}
	\end{matrix}\right]
	\end{align}
	where $\mathbf{q}= \mathbf{b}-\matc'\mathbf{a}^{(-)} \matc $ and $(.)^{(-)}$ is the Moore-Penrose generalized inverse as before.  Now if $ \m ={\xx_\II}' \dmat \xx_\II$ then $\mathbf{q}=\tx'\dmat_{11}\tx - \tx'\dmat_{10}\tx \left(\tx'\dmat_{00}\tx \right)^{(-)} \tx'\dmat_{01}\tx$. But because of Lemma \ref{LA1}, and using the definition of generalized inverse, it follows that $\mathbf{q}=0$. So the inverse reduces to
	\begin{align*}			
	\mathbf{m}^{(g)}=& \left[\begin{matrix}
	\mathbf{a}^{(-)} & 0 \\ 0 & 0 
	\end{matrix}\right]
	\\ =&\left[\begin{matrix}
	(\tx ' \dmat_{00} \tx )^{(-)} & 0 \\ 0 & 0 
	\end{matrix}\right].
	\end{align*}
	Therefore, 
	\begin{align}\label{b.optg}
	\left(\xx_\II ' \dmat \xx_\II \right)^{(g)}\xx_\II'\dmat y=&
	\left[\begin{matrix}
	(\tx ' \dmat_{00} \tx )^{(-)} & 0 \\ 0 & 0  	\end{matrix} \right]
	\left[\begin{matrix}
	\tx'\dmat_{00}y_0-\tx'\dmat_{01} y_1 \\  -\tx'\dmat_{10}y_0+\tx'\dmat_{11} y_1
	\end{matrix} \right] \nonumber
	\\=&
	\left[\begin{matrix}
	(\tx ' \dmat_{00} \tx )^{(-)}\left(\tx'\dmat_{00}y_0-\tx'\dmat_{01} y_1 \right)\\ 0
	\end{matrix} \right]. 
	\end{align}
	Next, using (\ref{b.optg}) and Lemma \ref{lemma.allbopt}, 
	\begin{align}\label{b.optgZ}
	b^{opt,gz}_\II= &
	 \left(\xx ' \dmat \xx \right)^{(g)}\xx'\dmat y + \left(\I_l-
	\left[\begin{matrix}
	(\tx ' \dmat_{00} \tx )^{(-)} & 0 \\ 0 & 0  	\end{matrix} 
	\right] 
	\left[\begin{matrix} \tx' \dmat_{00}\tx & \tx'\dmat_{01}\tx \\ \tx'\dmat_{10}\tx & \tx'\dmat_{11}\tx \end{matrix}\right]
	\right)z \nonumber
	\\= &
	 \left[\begin{matrix}
	 (\tx ' \dmat_{00} \tx )^{(-)}\left(\tx'\dmat_{00}y_0-\tx'\dmat_{01} y_1 \right)\\ 0
	 \end{matrix} \right] +
	\left[\begin{matrix}
	\I_{\scriptscriptstyle (k+1)}-(\tx ' \dmat_{00} \tx )^{(-)}\tx' \dmat_{00}\tx & (\tx ' \dmat_{00} \tx )^{(-)}\tx' \dmat_{01}\tx \\ 0 & \I_{\scriptscriptstyle (k+1)}  	\end{matrix} 
	\right]  z . 
	\end{align}
	Finally letting $z=\left[\begin{matrix} 0 \\ (\tx ' \dmat_{11} \tx )^{(-)} \tx \dmat_{11} y_1 \end{matrix}\right]$ leads to the result, with the last steps requiring the use of the reflexive property of the Moore-Penrose generalized inverse (i.e., for a symmetric matrix $\mathbf{a}$, $\mathbf{a}^{(-)}\mathbf{a}\mathbf{a}^{(-)}=\mathbf{a}^{(-)}$) and Lemma \ref{LA1}.  
\end{proof}

\begin{remark}
	It is not always the case that an optimal coefficient vector has a ``separable" solution.  It can be the case that, in order to be optimal, the subvector of the coefficient associated with treatment potential outcomes must take account of control potential outcomes and vice versa. Surprisingly, this can be true even under the sharp null.
\end{remark}


Next, two additional lemmas will be necessary before showing the optimality of the OLS coefficient. Lemma \ref{LemmaDeMean} will show that, for completely randomized experiments, multiplying $\dmat_{11}$, $\dmat_{00}$, $\dmat_{01}$, or $\dmat_{10}$ by a length-$n$ column vector zero-centers the vector and rescales by a constant. Lemma \ref{LemmaCovar} will show that matrices such as $\tx'\dmat_{11}\tx$ and $\tx'\dmat_{11}{y_{1}}$ represent finite-population covariance matrices rescaled by constants.

\begin{lemma}\label{LemmaDeMean}
	In a completely randomized experiment, $\dmat_{11} \tx =\frac{n n_0}{(n-1)n_1}(\tx-1_{\scriptscriptstyle n} \mu_{\tx} )$, with $1_{\scriptscriptstyle n}$ as a $(n\times 1)$ vector of ones and $\mu_{\tx}$ a $k+1$ rowvector giving the column means of $\tx$.  Likewise, in a completely randomized experiment, $\dmat_{00}\tx=\frac{n n_1}{(n-1)n_0}(\tx-1_{\scriptscriptstyle n} \mu_{\tx} )$.  And $\dmat_{10}\tx=\dmat_{01}\tx=-\frac{n}{n-1}(\tx-1_{\scriptscriptstyle n} \mu_{\tx} )$.
\end{lemma}
\begin{proofatend}
	By the definition of $\dmat_{11}$ above, in a completely randomized design the diagonal elements of $\dmat_{11}$ are 
	\begin{align*}
	\frac{\pi_{1i}-\pi_{1i}\pi_{1i}}{\pi_{1i}\pi_{1i}}=&\frac{\frac{n_1}{n}-\frac{n_1}{n}\frac{n_1}{n}}{\frac{n_1}{n}\frac{n_1}{n}}
	\\=&\frac{n-n_1}{n_1}
	\\ =&\frac{n_0}{n_1}
	\end{align*}
	and off-diagonal elements
	\begin{align*}
	\frac{\pi_{1i1j}-\pi_{1i}\pi_{1j}}{\pi_{1i}\pi_{1j}}=&\frac{\frac{n_1}{n}\frac{n_1-1}{n-1}-\frac{n_1}{n}\frac{n_1}{n}}{\frac{n_1}{n}\frac{n_1}{n}}
	\\=&-\frac{1}{n-1}\frac{n_0}{n_1}.
	\end{align*}
	So if we define
	\begin{align*}
	\dmat^*_{11}=\frac{n_1(n-1)}{n_0 n}\dmat_{11}
	\end{align*}
	then $\dmat^*_{11}$ has diagonal elements $\frac{n-1}{n}$ and off-diagonals $-\frac{1}{n}$, so that we can see that $\dmat^*_{11} \tx = \tx - 1_n\mu_{\tx}$ returns the de-meaned $\tx$. Therefore, $\dmat_{11}$ is a matrix that, when post-multiplied by $\tx$, returns a de-meaned $\tx$ that has been multiplied by the constant $\frac{nn_0}{(n-1)n_1}$. The proofs for $\dmat_{00}\tx$, $\dmat_{01}\tx$ and $\dmat_{01}\tx$ are analogous.
\end{proofatend}
\begin{proof}
	Provided in Appendix.
\end{proof}

\begin{corollary}\label{corollaryDeMean}
	For any constant vector, $c_{\scriptscriptstyle n}$, $\dmat_{00} c_{\scriptscriptstyle n}= \dmat_{11} c_{\scriptscriptstyle n}= \dmat_{01} c_{\scriptscriptstyle n}= \dmat_{10} c_{\scriptscriptstyle n} =0_{\scriptscriptstyle n}$.
\end{corollary}

\begin{lemma}\label{LemmaCovar}
	In a completely randomized design
	\begin{align*}
	\tx'\dmat_{11}\tx
	=& c_{11}\var(\tx)
	\\ \tx'\dmat_{00}\tx=& c_{00}\var(\tx),
	\\		\tx'\dmat_{01}\tx=& c_{01}\var(\tx),
	\\		\text{and } \tx'\dmat_{10}\tx=& c_{10}\var(\tx),
	\end{align*}
	where $\var(\tx):=n^{-1}\sum_i(\tx_i-\mu_{\tx})(\tx_i-\mu_{\tx})'$ is the finite-population variance-covariance matrix for $\tx$, $c_{11} := \frac{n^2 n_0}{(n-1)n_1}$, $c_{00} := \frac{n^2 n_1}{(n-1)n_0}$, and $c_{01}=c_{10} := -\frac{n^2}{(n-1)}$.
	Similarly, 
	\begin{align*}
	\tx'\dmat_{11} y_1
	 =&c_{11}\cov(\tx, y_1)
	 \\ 	\tx'\dmat_{00}y_0=& c_{00}\cov(\tx, y_0),
	 \\		\tx'\dmat_{01} y_1=& c_{01}\cov(\tx, y_1),
	 \\		\text{and } \tx'\dmat_{10}y_0=& c_{10}\cov(\tx, y_0).
	\end{align*}
	where, for example, $\cov(\tx, y_1):=n^{-1}\sum_i(\tx_i-\mu_{\tx})(y_{1i}-\mu_{y_1})'$ is a vector of finite-population covariances between $y_1$ and $x$'s.
\end{lemma}
\begin{proof}
	Results follow from Lemma \ref{LemmaDeMean} and the fact that $\sum_i(\tx_i-\mu_{\tx})\tx_i'=\sum_i(\tx_i-\mu_{\tx})(\tx_i-\mu_{\tx})'$.
\end{proof}

Finally, the next two theorems present the main results of the section.
\begin{theorem}
	In a completely randomized design with specification \emph{II}, the OLS coefficient given in Definition \ref{def.ols} minimizes the variance of the fixed-coefficient generalized regression estimator, i.e., $b^{ols}_\II$ is in the set of optimal fixed-coefficients defined in Lemma \ref{lemma.allbopt}.
\end{theorem}
\begin{proof}
	Using Lemma \ref{LemmaCovar} write
	\begin{align*}
	\left(\tx' \dmat_{11} \tx \right)^{(-)}\tx' \dmat_{11} y_1 =& 	\left[\begin{matrix}
	 0 & 0_{\scriptscriptstyle k}'\\ 0_{\scriptscriptstyle  k } & \var(\x)
	\end{matrix}
	\right]^{(-)}\left[\begin{matrix}
	0  \\ \cov(\x, y_1)
	\end{matrix}
	\right]
		\\ =& \left[\begin{matrix}
		0  \\ \var(\x)^{(-)}  \cov(\x, y_1)
		\end{matrix}
		\right]
	\end{align*}
And note that unless the columns of $\x$ are colinear, $(.)^{(-)}$ is equivalent to the usual inverse. Now use Lemma \ref{lemma.allbopt} and let
\begin{align*}
z=& \left[\begin{matrix}
\mu_{y_1}  \\ 0_{\scriptscriptstyle k }
\end{matrix}
\right],
\end{align*} 
where $\mu_{y_1}$ is he mean of treatment potential outcomes, to arrive at an optimal sub-vector for treatment potential outcomes is
\begin{align*}
& \left[\begin{matrix}
\mu_{y_1}   \\ \var(\x)^{-1}  \cov(\x, y_1)
\end{matrix}
\right].
\end{align*}
An analogous optimal sub-vector for control potential outcomes can be defined. The result follows.
\end{proof}



\begin{theorem}\label{b.ols.est.isOpt}
	Under Assumptions 1 and 2, in a completely randomized design with specification \emph{II}, the OLS coefficient estimator given in Definition \ref{def.ols.est} has a conjugate ATE estimator that obtains asymptotic minimum variance within the class of generalized regression estimators. 
\end{theorem}
\begin{proofatend}
To see that $\widehat{b}^{ols}_\II$ estimates $b^{ols}_\II$ note that
\begin{align*}
\E \left[{\xx_\II}' \R \xx_\II \right] = & \left[\begin{matrix}
-1'_{\scriptscriptstyle n} & 0'_{\scriptscriptstyle n}
\\ -\x' &  0'_{\scriptscriptstyle (\frac{k}{2}-1) \times n}
\\ 0'_{\scriptscriptstyle n} & 1'_{\scriptscriptstyle n}
\\ 0'_{\scriptscriptstyle (\frac{k}{2}-1) \times n} & \x'
\end{matrix}\right]
\left[\begin{matrix}
-\frac{n_0}{n} 1_{\scriptscriptstyle n} & -\frac{n_0}{n}\x & 0_{\scriptscriptstyle n} & 0_{\scriptscriptstyle (\frac{k}{2}-1) \times n}
\\ 0_{\scriptscriptstyle n} & 0_{\scriptscriptstyle (\frac{k}{2}-1) \times n} & \frac{n_1}{n} 1_{\scriptscriptstyle n} & \frac{n_1}{n} \x
\end{matrix}\right]
\\ = & \left[\begin{matrix}
n_0 & 0 & 0 & 0 
\\ 0 & n_0 \var(\x) & 0 & 0 
\\ 0 & 0 & n_1 & 0 
\\ 0 & 0 & 0 & n_1 \var(\x)
\end{matrix}\right]
\end{align*}
and 
\begin{align*}
\E \left[{\xx_\II}' \R y \right] = & \left[\begin{matrix}
-1'_{\scriptscriptstyle n} & 0'_{\scriptscriptstyle n}
\\ -\x' &  0'_{\scriptscriptstyle (\frac{k}{2}-1) \times n}
\\ 0'_{\scriptscriptstyle n} & 1'_{\scriptscriptstyle n}
\\ 0'_{\scriptscriptstyle (\frac{k}{2}-1) \times n} & \x'
\end{matrix}\right]
\left[\begin{matrix}
-\frac{n_0}{n} y_0 \\ \frac{n_1}{n} y_1 
\end{matrix}\right]
\\ = & \left[\begin{matrix}
n_0 \mu_{y_0} \\ n_0 \cov(\x, y_0) \\n_1 \mu_{y_1} \\ n_1 \cov(\x, y_1) 
\end{matrix}\right]
\end{align*}
so that 
\begin{align*}
\E \left[{\xx_\II}' \R \xx_\II \right]^{-1}\E \left[{\xx_\II}' \R y \right] = & \left[\begin{matrix}
n_0 & 0 & 0 & 0 
\\ 0 & n_0 \var(\x) & 0 & 0 
\\ 0 & 0 & n_1 & 0 
\\ 0 & 0 & 0 & n_1 \var(\x)
\end{matrix}\right]^{-1}\left[\begin{matrix}
n_0 \mu_{y_0} \\ n_0 \cov(\x, y_0) \\n_1 \mu_{y_1} \\ n_1 \cov(\x, y_1) 
\end{matrix}\right]
\\ = & \left[\begin{matrix}
n_0^{-1} & 0 & 0 & 0 
\\ 0 & n_0^{-1} \var(\x)^{-1} & 0 & 0 
\\ 0 & 0 & n_1^{-1} & 0 
\\ 0 & 0 & 0 & n_1^{-1} \var(\x)^{-1}
\end{matrix}\right]^{-1}\left[\begin{matrix}
n_0 \mu_{y_0} \\ n_0 \cov(\x, y_0) \\n_1 \mu_{y_1} \\ n_1 \cov(\x, y_1) 
\end{matrix}\right]
\\ =& \left[\begin{matrix}
\mu_{y_0} \\ \var(\x)^{-1} \cov(\x, y_0) \\ \mu_{y_1} \\ \var(\x)^{-1} \cov(\x, y_1) 
\end{matrix}\right]
\end{align*}
Hence under suitable regularity conditions $\widehat{b}^{ols}_\II\rightarrow {b}^{ols}_\II$ so that $\widehat{\delta}^{\tr,ols}_\II$ is asymptotically optimal.
\end{proofatend}
\begin{proof}
	Provided in Appendix.
\end{proof}


\subsection{Tyranny of the minority is optimal for completely randomized designs with specification I}\label{section.tyranny.opt}

In this section, it will be shown that an optimal coefficient for the fixed-coefficient generalized regression estimator for completely randomized designs and specification I is the ``tyranny of the minority" coefficient, call it $b^{tyr}_\sI$, and a WLS estimator of the coefficient will be defined.  It is noteworthy that, by contrast, there is no OLS analogue that is generally asymptotically optimal for specification I for completely randomized experiments. The section will also show that tyranny of the minority can achieve asymptotic precision using specification I that is as good as optimal estimators that use specification II. 

First define the tyranny of the minority coefficient for specification I and its estimator.

\begin{definition}[Tyranny of the minority coefficient]\label{def.tyr}
	The ``tyranny of the minority" coefficient for specification \emph{I} is given by
	\emph{\begin{align*}
		b^{tyr}_\sI:=&\left({\xx_\sI}'  \left(\I_{\scriptscriptstyle 2n}-\bpi\right) \xx_\sI \right)^{-1} {\xx_\sI}' \left(\I_{\scriptscriptstyle 2n}-\bpi\right) y
		\\=&\left[\begin{matrix}
		\mu_{y_0} \\ 
		\mu_{y_1} \\
		\frac{n_1}{n} \var(\x)^{(-)}\cov(\x, y_0) + \frac{n_0}{n} \var(\x)^{(-)}\cov(\x, y_1)
		\end{matrix}\right]
		\end{align*}}where $\mu_{y_0}$ and $\mu_{y_1}$ are means of control and treatment potential outcomes, respectively, and $\I_{\scriptscriptstyle 2n}$ is a $2n \times 2n$ identity matrix.
\end{definition}

Note in the that $\var( \x )^{(-)}\cov(\x, y_0)$ is the population least squares coefficients when regressing $y_0$ on $\x$, and $\var(\x)^{(-)}\cov(\x, y_1)$ is, likewise, the population least squares coefficients when regressing $y_1$ on $\x$. The weights for combining these two components, $\frac{n_1}{n}$ and $\frac{n_0}{n}$, respectively, are such that the coefficient for the arm with fewer units gets more weight. Hence, the name ``tyranny of the minority".

\begin{definition}[Tyranny of the minority coefficient estimator]\label{def.tyr.est} The ``tyranny of the minority coefficient estimator" for specification \emph{I} is
	\emph{\begin{align*}
		\widehat{b}^{tyr}_\sI=& \left({\xx_\sI}' \R \left(\bpi^{-1}- \I_{\scriptscriptstyle 2n}\right) \xx_\sI \right)^{-1} {\xx_\sI}' \R \left(\bpi^{-1}- \I_{\scriptscriptstyle 2n}\right) y.
		\end{align*}} where $\I_{\scriptscriptstyle 2n}$ is a $2n \times 2n$ identity matrix. 	
\end{definition}

To prove that $b^{tyr}_\sI$ in Definition \ref{def.tyr} is an optimal choice of coefficient for the fixed-coefficient generalized regression estimator, first define an equivalent coefficient for specification II.

\begin{definition}[Tyranny of the minority coefficient for specification II]\label{def.tyr.II}
	The ``tyranny of the minority coefficient for specification \emph{II}" is 
	\emph{\begin{align*}
		b^{tyr}_\II = \left[\begin{matrix}
		\mu_{y_0} \\ \frac{n_1}{n}\var({\x})^{-1}\cov({\x}, y_0)+\frac{n_0}{n}\var({\x})^{-1}\cov(\x, y_1)
		\\ 		 \mu_{y_1} \\ \frac{n_1}{n}\var({\x})^{-1}\cov( {\x}, y_0)+\frac{n_0}{n}\var({\x})^{-1}\cov({\x}, y_1)
		\end{matrix} \right].
		\end{align*}}
\end{definition}

Comparing Definition \ref{def.tyr.II} to Definition \ref{def.tyr} reveals that the ``slope" coefficients, given by $\frac{n_1}{n}\var(\x)^{-1}\cov(\x, y_0)+\frac{n_0}{n}\var(\x)^{-1}\cov(\x, y_1)$, are identical for the two specifications.  The implication is that $\xx_\sI b^{tyr}_\sI=\xx_\II b^{tyr}_\II$ and hence, the conjugate ATE estimators are algebraically equivalent. Therefore, if $b^{tyr}_\II$ is in the set of optimal choices for a fixed-coefficient in specification II, then $b^{tyr}_\sI$ must be among the optimal coefficients for the fixed-coefficient generalized regression estimator for specification I.

\begin{theorem}
	For completely randomized experiments with specification \emph{II}, the tyranny of the minority coefficient given in Definition \ref{def.tyr.II} is an optimal coefficient for the fixed-coefficient generalized regression estimator.
\end{theorem}
\begin{proof}
	Beginning with Lemma \ref{lemma.allbopt} and again arriving at equation (\ref{b.optgZ}), this time let
	\begin{align*}
	z = & \left[\begin{matrix} \mu_{y_0} \\ 0_{\scriptscriptstyle k} \\  \mu_{y_1}  \\  \frac{n_1}{n}\var(\x)^{-1}\cov(\x, y_0)+\frac{n_0}{n}\var(\x)^{-1}\cov(\x, y_1)  \end{matrix} \right].
	\end{align*}
	The result follows.
\end{proof}
\begin{corollary}
	For completely randomized experiments with specification \emph{I}, the tyranny of the minority coefficient given in Definition \ref{def.tyr} is an optimal coefficient for the fixed-coefficient generalized regression estimator.
\end{corollary}

\begin{theorem}
	Under Assumptions 1 and 2, in a completely randomized design with specification \emph{I}, the tyranny of the minority coefficient estimator given in Definition \ref{def.tyr.est} has a conjugate ATE estimator that obtains asymptotic minimum variance within the class of generalized regression estimators. 
\end{theorem}
\begin{proofatend}
	First,
	\begin{align*}
	\E\left[\xx_\sI' \R \left( \bpi^{-1}-\I_{\scriptscriptstyle 2n} \right) \xx_\sI \right]=& \xx_\sI' \bpi \left( \bpi^{-1}-\I_{\scriptscriptstyle 2n}\right) \xx_\sI
	\\=& \xx_\sI' \left( \I_{\scriptscriptstyle 2n}-\bpi\right) \xx_\sI
	\end{align*}
	and 
	\begin{align*}
	\E\left[\xx_\sI' \R \left( \bpi^{-1}-\I_{\scriptscriptstyle 2n} \right) y \right]=& \xx_\sI' \bpi \left( \bpi^{-1}-\I_{\scriptscriptstyle 2n}\right) y
	\\=& \xx_\sI' \left( \I_{\scriptscriptstyle 2n}-\bpi\right) y
	\end{align*}
	so that 
	\begin{align*}
	\E\left[\xx_\sI' \R \left( \bpi^{-1}-\I_{\scriptscriptstyle 2n} \right) \xx_\sI \right]^{-1}\E\left[\xx_\sI' \R \left( \bpi^{-1}-\I_{\scriptscriptstyle 2n} \right) y \right]= & \left(\xx_\sI' \left( \I_{\scriptscriptstyle 2n}-\bpi\right) \xx_\sI\right)^{-1} \xx_\sI' \left( \I_{\scriptscriptstyle 2n}-\bpi\right) y
	\end{align*}
	which is just the coefficient given in Definition \ref{def.tyr}.  Thus, under suitable regularity conditions $\widehat{b}^{tyr}_\sI \rightarrow b^{tyr}_\sI$ so that its conjugate ATE estimator is asymptotically optimal.
\end{proofatend}
	\begin{proof}
		Provided in Appendix.
	\end{proof}
	

	\subsection{OLS coefficients minimize AS bound for completely randomized designs with specification II}
	
	Given that OLS with specification II (see Section \ref{section.ols.opt}) and the tyranny of the minorty estimator with specification I (see Section \ref{section.tyranny.opt}) can be equally precise, it is unclear which might be preferable.  One way to evaluate this is to see which leads to a tighter variance bound.  In this section, it is shown that in completely randomized designs and specification II, the coefficients that minimize the AS variance bound are given by $b^{ols}_\II$ in Definition \ref{def.ols}.  The result suggests that, when using the AS bound, OLS with specification II will tend to lead to smaller intervals than tyranny of the minority.
	
	\begin{theorem}\label{proof.ols.specII}
	In the completely randomized design with specification \emph{II}, if we let $u=y-\xx_\II b^f_\II$ with $b^f_\II$ being a fixed-coefficient, then a value of $b^f_\II$ that minimizes the bound on the variance, $n^{-2}u' \tilde{\dmat} u$, is \emph{$b^{ols}_\II$} from Definition \ref{def.ols}.
	\end{theorem}
	\begin{proofatend} First, let $\xx_{\II}^*$ be an equivalent specification to specification II defined as 
	\begin{align*}
	\xx_{\II}^*= & \left[
	\begin{matrix}
	-1_{\scriptscriptstyle n} & 0_{\scriptscriptstyle n} & -\x & 0_{(\scriptscriptstyle n \times {\scriptscriptstyle k} )}
	\\0_{\scriptscriptstyle n}& 1_{\scriptscriptstyle n}&  0_{(\scriptscriptstyle n \times {\scriptscriptstyle k})} & \x 
	\end{matrix} \right].
	\end{align*}
	Then,
	\begin{align*}
	{\xx_{\II}^*}' \tilde{\dmat} \xx_{\II}^* = & {\xx_{\II}^*}' {\dmat} \xx_{\II}^* + {\xx_{\II}^*}' \left[ \begin{matrix}
	\I & \I \\ \I & \I 
	\end{matrix}\right]\xx_{\II}^*
	\\ = & \left[ \begin{matrix}
	n & -n & 0 & 0 
	\\ -n & n & 0 & 0
	\\ 0 & 0 & \x' \tilde{\dmat}_{00} \x & -\x' \tilde{\dmat}_{01} \x
	\\ 0 & 0 & -\x' \tilde{\dmat}_{10} \x & \x' \tilde{\dmat}_{11} \x
	\end{matrix} \right]
	\end{align*}
	Now recall that $\x$ is zero-centered and using Lemma \ref{LemmaCovar} we have for a completely randomized design
		\begin{align*}
			\x' \tilde{\dmat}_{00} \x =& \x' {\dmat}_{00} \x + \x'\x
			\\ =& \frac{n^2 n_1}{ (n-1)n_0}\var(\x)+n \var(\x)
			\\ = & c_a \var(\x)
		\end{align*}
	where $c_a :=\frac{n^2 n_1+ n(n-1)n_0}{ (n-1)n_0}$.  Likewise,
		\begin{align*}
		\x' \tilde{\dmat}_{11} \x = & c_b \var(\x),
		\\ -\x' \tilde{\dmat}_{01} \x = & c_c \var(\x),
		\\ \text{and} \hspace{2mm} 	-\x' \tilde{\dmat}_{10} \x = & c_c \var(\x)
		\end{align*}
	with $c_b := \frac{n^2 n_0+ n(n-1)n_1}{ (n-1)n_1} $ and $c_c := \frac{-n^2+n-1}{(n-1)}$. Next, letting $c_q := c_b -c_c^2 c_a^{-1}$ and given that a generalized inverse of a partitioned matrix is given in (\ref{GenInv}),
	\begin{align*}
	\left({\xx^*_{\II}}' \tilde{\dmat} \xx^*_\II \right)^{(g)} = \text{Bdiag}\left(
	\left[\begin{matrix}
	n & -n \\ -n & n
	\end{matrix}\right]^{(-)}, \left[\begin{matrix}
	c_a^{-1}+c_a^{-2}c_c^2 c_q^{-1} & -c_a^{-1} c_c c_q^{-1} \\ -c_a^{-1} c_c c_q^{-1} & c_q^{-1}
	\end{matrix}\right] \otimes \var(\x)^{(-)} \right)
	\end{align*}
	where $\text{Bdiag}\left( \mathbf{a}, \mathbf{b}\right)$ makes a block diagonal matrix out of matrices $\mathbf{a}$ and $\mathbf{b}$ and $\otimes$ is the Kronecker product. 
	Similarly, 
	\begin{align*}
	\x' \tilde{\dmat}_{00} y_0 =& c_a \cov(\x, y_0)
	\\ \x' \tilde{\dmat}_{11} {y_1} = & c_b \cov(\x, y_1),
	\\ -\x' \tilde{\dmat}_{01} y_1 = & c_c \cov(\x, y_1),
	\\ \text{and} \hspace{2mm} -\x' \tilde{\dmat}_{10} y_0 = & c_c \cov(\x, y_0),
	\end{align*}
	so that
	\begin{align*}
	{\xx_\II}' \tilde{\dmat} y= \left[\begin{matrix}
	-1'_{\scriptscriptstyle 2n}y \\  1'_{\scriptscriptstyle 2n}y \\
	\left[\begin{matrix}
	 c_a \\ c_c \end{matrix} \right] \otimes \cov(\x, y_0) +\left[\begin{matrix}
	 c_c \\ c_b \end{matrix} \right] \otimes \cov(\x, y_1)
	 \end{matrix} \right].
	\end{align*}
	Therefore,
	\begin{align*}
	\left({\xx^*_\II}' \tilde{\dmat} \xx^*_\II \right)^{(g)} {\xx^*_\II}' \tilde{\dmat} y =& \left[ \begin{matrix} \left[\begin{matrix}
	n & -n \\ -n & n
	\end{matrix}\right]^{(-)} \left[\begin{matrix}
	-1'_{\scriptscriptstyle 2n}y \\  1'_{\scriptscriptstyle 2n}y \end{matrix} \right] 
	\\ \left( \left[\begin{matrix}
	c_a^{-1}+c_a^{-2}c_c^2 c_q^{-1} & -c_a^{-1} c_c c_q^{-1} \\ -c_a^{-1} c_c c_q^{-1} & c_q^{-1}
	\end{matrix}\right] \otimes \var(\x)^{(-)}\right)\left( \left[\begin{matrix}
	c_a \\ c_c \end{matrix} \right] \otimes \cov(\x, y_0)
	+\left[\begin{matrix}
	c_c \\ c_b \end{matrix} \right] \otimes \cov(\x, y_1)\right)
	\end{matrix} \right].
	\end{align*}
	Focusing on the last $2k$ coefficients we have,
	\begin{align*}
	\left( \left[\begin{matrix}
	c_a^{-1}+c_a^{-2}c_c^2 c_q^{-1} & -c_a^{-1} c_c c_q^{-1} \\ -c_a^{-1} c_c c_q^{-1} & c_q^{-1}
	\end{matrix}\right] \otimes \var(\x)^{(g)}\right)&\left( \left[\begin{matrix}
	c_a \\ c_c \end{matrix} \right] \otimes \cov(\x,y_0)
	+\left[\begin{matrix}
	c_c \\ c_b \end{matrix} \right] \otimes \cov(\x,y_1)\right) 
	\\ =& \left( \left[\begin{matrix}
	c_a^{-1}+c_a^{-2}c_c^2 c_q^{-1} & -c_a^{-1} c_c c_q^{-1} \\ -c_a^{-1} c_c c_q^{-1} & c_q^{-1}
	\end{matrix}\right]\left[\begin{matrix}
	c_a \\ c_c \end{matrix} \right] \right)\otimes \var(\x)^{(-)} \cov(\x,y_0)
	\\&+\left( \left[\begin{matrix}
	c_a^{-1}+c_a^{-2}c_c^2 c_q^{-1} & -c_a^{-1} c_c c_q^{-1} \\ -c_a^{-1} c_c c_q^{-1} & c_q^{-1}
	\end{matrix}\right]\left[\begin{matrix}
	c_c \\ c_b \end{matrix} \right] \right)\otimes \var(\x)^{(-)} \cov(\x,y_1)
	\\ =& \left[\begin{matrix}
	1 \\ 0 \end{matrix} \right] \otimes \var(\x)^{(-)} \cov(\x,y_0) 
	+\left[\begin{matrix}
	0\\ 1 \end{matrix} \right] \otimes \var(\x)^{(-)} \cov(\x,y_1) 
	\\ =& \left[\begin{matrix}
	\var(\x)^{(-)} \cov(\x,y_0) \\ \var(\x)^{(-)} \cov(\x,y_1) \end{matrix} \right].
	\end{align*}
	The first equality follows from the mixed-product property of Kronecker products. The following line applies algebra and the definition of $c_q$.  As long as there is no perfect collinearity in $\x$, $\var(\x)^{(-)}$ represents the usual inverse matrix.  
	The intercept coefficients are
	\begin{align*}
	\left[\begin{matrix}
	n & -n \\ -n & n
	\end{matrix}\right]^{(-)} \left[\begin{matrix}
	-1'_{\scriptscriptstyle 2n}y \\  1'_{\scriptscriptstyle 2n}y \end{matrix} \right]=\frac{1}{2} \left[\begin{matrix}
	-\delta \\  \delta
	\end{matrix}\right], 
	\end{align*}
	but recognizing that the choice of generalized inverse was arbitrary, it can be seen that the full range of optimal intercepts includes 
	\begin{align*}
	\left[\begin{matrix}
	\mu_{y_0} \\ \mu_{y_1} 
	\end{matrix}\right].
	\end{align*}
\end{proofatend}
\begin{proof}
	Provided in Appendix.
\end{proof}

\subsection{{\rara} is algebraically equivalent to OLS for completely randomized designs with specification II}

It has been shown that OLS is asymptotically optimal in completely randomized experiments. In this section, it is demonstrated that the two-step optimal estimator, $\widehat{\delta}^{\tr,\trara}$, is algebraically equivalent to the OLS estimator, $\widehat{\delta}^{\tr,ols}$.  
 
\begin{theorem}
	The vector of residuals, $\R_1 \bpi^{-1}\widehat{u}_1$, is orthogonal to the weights $\dmat_{11} \tx (\tx'\dmat_{11} \tx)^{(g)}$.
\end{theorem}
\begin{proof}
From the lemmas above we have
\begin{align*}
\tx'\dmat_{11} \R_1 \bpi^{-1} \widehat{u}_1= &c\sum_i (\tx_i-\mu_x)\widehat{u}_{1i} R_i
\\ =&c \times 0
\end{align*}
where $c=\frac{n}{n_1}\frac{n^2}{n_1}\left(1-\frac{n_1-1}{n-1}\right)$ (with the first $\frac{n}{n_1}$ coming from $\bpi^{-1}$). The last line follows from the fact that we know that for OLS that the column space of (the observed) $\tx$'s is orthogonal to the residuals.
\end{proof}
The result shows that the two-step optimal will not make any adjustment to the OLS estimates in the completely randomized case.  The estimators are algebraically equivalent. 

\section{Optimal Regression for Cluster-Randomization}\label{section.cluster}

This section reports on results for experiments with complete randomization of clusters. As above, it is assumed that we have an identified design and that every arm has at least two units of randomization (clusters) assigned to it.

It is also assumed that there is no second-stage selection from within clusters, which is to say that covariates and outcomes are available for all cluster members. Extensions that account for second-stage sampling are possible but beyond the scope of the paper.
	
When analyzing cluster randomized experiments, one approach to estimating the ATE is to regress the individual-level data on the treatment indicator and covariates using OLS, but it is not asymptotically optimal. By contrast, as the next subsections will show, regression using cluster totals is asymptotically optimal.\footnote{It should also be noted that an alternative approach is to first take cluster averages before running regression. This approach is biased and not generally consistent for the ATE \citep{middleton08}. However, if one were content to estimate the average of cluster-level average effects, this approach may be acceptable. There are benefits to doing this. For example, results from Section \ref{section.completelyrand} can be applied directly.  Moreover, compared to analyzing cluster totals, high leverage observations, which can foul normal-theory inference, are less likely.  Moreover, in the presence of treatment effects, summing to create cluster-level totals is likely to induce a correlation between the leverage of an observation and its treatment effect (in this case the sum of treatment effects for units in the cluster). The first-order term in regression's bias is the correlation between leverage and treatment effect \citep{lin, freedman08a, freedman08b}.}

\subsection{OLS with cluster totals is optimal for cluster-randomized designs with specification II}

In this subsection, it will be shown that regression with cluster totals is asymptotically optimal for cluster randomized experiments using specification II. 

First, let $m$, $m_0$, and $m_1$ be the number of clusters, the number of clusters in treatment and the number of clusters in control, respectively. Meanwhile, let $c_i$ give the cluster id number for the cluster that includes unit $i$, and let $\tx^c_{\scriptscriptstyle n}$ represent an $n \times (k+1)$ matrix of cluster totals, i.e., with $i^{th}$ row giving the sum of rows in $\x$ associated with units in cluster $c_i$. By contrast, let $\tx^c$ represent an $m \times (k+1)$ matrix of cluster totals, with $g^{th}$ row giving the sum of rows of $\tx$ associated with units in cluster $g$.

\begin{definition}[OLS with cluster totals]\label{def.ols.w.totals} The "OLS with cluster totals" coefficient for specification \emph{II} is 
	\emph{\begin{align*}
		b^{ols,c}_\II = \left[\begin{matrix} 
		\var({\tx^c})^{(-)} \cov({\tx^c}, y^c_0) \\ \var(	{\tx^c})^{(-)} \cov({\tx^c}, y^c_1)
		\end{matrix} \right]
		\end{align*}}where ${\tx^c}$ is the $m \times (k+1)$ matrix with row $g$ including cluster totals for the $g^{th}$ cluster. Likewise, $y^c_0$ and $y^c_1$ are length $m$ with entry $g$ representing cluster totals for the $g^{th}$ cluster's $y_{0i}$ and $y_{1i}$ values, respectively.
\end{definition}

Next, to define the corresponding coefficient estimator, first let specification II$^c$ be as follows

\begin{align*}
\xx^c_{\II}= \left[\begin{matrix}
-1_{\scriptscriptstyle m} &   0_{\scriptscriptstyle m} & -\tx^c & 0_{\scriptscriptstyle m \times (k+1)}
\\ 0_{\scriptscriptstyle m}  & 1_{\scriptscriptstyle m} & 0_{\scriptscriptstyle m \times (k+1)} & \tx^c 
\end{matrix}\right].
\end{align*}
Note $\xx^c_{\II}$ has $2k+4$ columns where $\xx_{\II}$ has only $2k+2$. 

\begin{definition}[OLS with cluster totals coefficient estimator]\label{def.ols.w.totals.est}
	The ``OLS with cluster totals coefficient estimator" for specification \emph{II} is
	\emph{\begin{align*}
		\left[ \begin{matrix}
		\widehat{a}_0 \\ \widehat{a}_1 \\ \widehat{b}^{ols,c}_{\II}
		\end{matrix}
		\right] = \left({\xx^c_{\II}}' \R^c  {\xx^c_{\II}} \right)^{-1} {\xx^c_{\II}}' \R^c y^c
		\end{align*}}
	where $\widehat{a}_0$ and $\widehat{a}_1$ are scalars and \emph{$\widehat{b}^{ols,c}_{\II}$} has length $2k+2$ and $\R^c$ (an analog to $\R$) is a $2m\times 2m$ diagonal matrix with cluster-level assignment indicators on the diagonal.
\end{definition}

The two lemmas that follow will lead into the final result of the section. Lemma \ref{lemma.clustertotals} shows that, for cluster-randomized experiments, multiplying $\dmat_{11}$, $\dmat_{00}$, $\dmat_{01}$, or $\dmat_{10}$ by a length-$n$ column vector returns a length-$n$ vector of cluster totals, zero-centered and multiplied by a constant. Lemma \ref{lemma.clustertotalsII} will show that matrices such as $\tx'\dmat_{11}\tx$ and $\tx'\dmat_{11}{y_{1}}$ represent finite-population covariance matrices for cluster totals rescaled by constants.

\begin{lemma}\label{lemma.clustertotals}
	In a cluster randomized experiment, $\dmat_{11}\tx=\frac{m m_0}{(m-1)m_1}(\tx^c_{\scriptscriptstyle n}- \frac{n}{m}1_{\scriptscriptstyle n} \mu_{\tx})$ where $\frac{n}{m}1_{\scriptscriptstyle n} \mu_{\tx}$ is a matrix that subtracts off the average cluster totals.  
	Likewise, in a cluster-randomized experiment, $\dmat_{00}\tx=\frac{m m_1}{(m-1)m_0}(\tx^c_{\scriptscriptstyle n}-\frac{n}{m}1_{\scriptscriptstyle n} \mu_{\tx})$.  And $\dmat_{10}\tx=\dmat_{01}\tx=-\frac{m}{m-1}(\tx^c_{\scriptscriptstyle n}-\frac{n}{m}1_{\scriptscriptstyle n} \mu_{\tx})$. 
\end{lemma}
\begin{proofatend}
	By the definition of $\dmat_{11}$ above, in a cluster randomized designs the $ij$ element of $\dmat_{11}$ when units $i$ and $j$ are in the same cluster is
	\begin{align*}
	\frac{\pi_{1i1j}-\pi_{1i}\pi_{1j}}{\pi_{1i}\pi_{1j}}=&\frac{\frac{m_1}{m}-\frac{m_1}{m}\frac{m_1}{m}}{\frac{m_1}{m}\frac{m_1}{m}}
	\\=&\frac{m-m_1}{m_1}
	\\ =&\frac{m_0}{m_1}
	\end{align*}
	and for $i,j$ not in the same cluster
	\begin{align*}
	\frac{\pi_{1i1j}-\pi_{1i}\pi_{1j}}{\pi_{1i}\pi_{1j}}=&\frac{\frac{m_1}{m}\frac{m_1-1}{m-1}-\frac{m_1}{m}\frac{m_1}{m}}{\frac{m_1}{m}\frac{m_1}{m}}
	\\=&-\frac{1}{m-1}\frac{m_0}{m_1}.
	\end{align*}
	Now define
	\begin{align*}
	\dmat^*_{11}=\frac{m_1(m-1)}{m_0 m}\dmat_{11}
	\end{align*}
	then $\dmat^*_{11}$ has $i,j$ element equal to $\frac{m-1}{m}$ if $i$ and $j$ are in the same cluster and equal to $-\frac{1}{m}$ otherwise. So, $\dmat^*_{11}\tx$ returns a length $n$ vector $(\tx^c_{\scriptscriptstyle n}- \frac{n}{m}1_{\scriptscriptstyle n} \mu_{\tx})$ with the $i^{th}$ row of $\tx^c_n$ equal to the sums of $x$'s for cluster $c_i$ and with $\frac{n}{m}1_{\scriptscriptstyle n} \mu_{\tx}$ doing the work of subtracting off the average of cluster totals. Therefore, $\dmat_{11}\tx=\frac{m m_0}{(m-1)m_1}(\tx^c_{\scriptscriptstyle n}- \frac{n}{m}1_{\scriptscriptstyle n} \mu_{\tx})$. The proofs for $\dmat_{00}\tx$, $\dmat_{01}\tx$ and $\dmat_{01}\tx$ are analogous.
\end{proofatend}
\begin{proof}
	Provided in Appendix.
\end{proof}

\begin{lemma}\label{lemma.clustertotalsII}
	In a cluster randomized experiment, $\tx' \dmat_{11}\tx=\frac{m^2m_0}{(m-1)m_1}\var(\tx^c )$.
	Likewise, the $\tx' \dmat_{11}y_1=\frac{m^2m_0}{(m-1)m_1}\cov(\tx^c, y^c_1)$ where $y^c_1$ is an length-$m$ vector with the $g^{th}$ element representing cluster totals for the $g^{th}$ cluster's $y_{1i}$ values.  
\end{lemma}
\begin{proofatend}
	Write
	\begin{align*}
	\tx'\dmat_{11}\tx=&\frac{m m_0}{(m-1)m_1} \tx' \left(\tx^c_{\scriptscriptstyle n}- \frac{n}{m}1_{\scriptscriptstyle n} \mu_{\tx} \right)
	\\=&\frac{m m_0}{(m-1)m_1} {{\tx^{c\prime}}_{\scriptscriptstyle m}} \left(\tx^c_{\scriptscriptstyle m}- \frac{n}{m}1_{\scriptscriptstyle m} \mu_{\tx} \right)
	\\ = & \frac{m^2m_0}{(m-1)m_1} \var( \tx^c_{\scriptscriptstyle m})
	\end{align*}
	where $\tx^c_{\scriptscriptstyle m}$ is an $m \times (k-1)$ vector (one row per cluster) with the $g^{th}$ row representing cluster totals of the rows of $\tx$ associated with members of the $g^{th}$ cluster.  
\end{proofatend}
\begin{proof}
Provided in Appendix.
\end{proof}

\begin{theorem}
	For cluster randomized experiments, the OLS with cluster totals coefficient in Definition \ref{def.ols.w.totals} is optimal for the fixed-coefficient generalized regression estimator. 
\end{theorem}
\begin{proof}
	Since $\pi_{1i}$ is equal for all $i$ in a cluster randomized experiment, then using Lemma \ref{lemma.sep}, we have that the optimal solution includes the separated coefficients in equation (\ref{equation.sep}). So, by Lemma \ref{lemma.clustertotalsII}, for cluster randomized experiments
	\begin{align*}\label{equation.sep}
	b^{sep}_\II= &\left[\begin{matrix} 
	(\tx ' \dmat_{00} \tx )^{(-)} \tx'\dmat_{00} y_0 \\ (\tx ' \dmat_{11} \tx )^{(-)} \tx'\dmat_{11} y_1 
	\end{matrix} \right]
	\\=& \left[\begin{matrix} 
	\var(	{\tx^c})^{(-)} \cov({\tx^c}, y^c_0) \\ \var(	{\tx^c})^{(-)} \cov({\tx^c}, y^c_1)
	\end{matrix} \right].
	\end{align*}
\end{proof}
\begin{remark}
	Note that the ``intercept'' terms are no longer constants when $\tx$ is collapsed to $\tx^c$.  In a sense, the intercept terms now ``control" for cluster size in OLS with cluster totals.
\end{remark}

\begin{theorem}
	Under Assumptions 1 and 2, in a cluster-randomized design with specification \emph{II}, the OLS coefficient with cluster totals estimator given in Definition \ref{def.ols.w.totals.est} has a conjugate ATE estimator that obtains asymptotic minimum variance within the class of generalized regression estimators.
\end{theorem}
\begin{proof}
	Provided in Appendix.
\end{proof}

	\subsection{Tyranny of the minority with cluster totals is optimal for cluster randomized experiments with specification I}

In this section, it will be shown that an optimal coefficient for the fixed-coefficient generalized regression estimator for cluster-randomized designs and specification I is the ``tyranny of the minority with cluster totals" coefficient, call it $b^{tyr,c}_\sI$. A WLS estimator of the coefficient will be defined. The section will also show that tyranny of the minority with cluster totals can achieve asymptotic precision using specification I that is as good as optimal estimators that use specification II.

First define the tyranny of the minority coefficient for cluster totals for specification I and its estimator.

\begin{definition}[Tyranny of the minority with cluster totals coefficient]\label{def.tyr.c}
	The ``tyranny of the minority" with cluster totals coefficient for specification \emph{I} is given by
	\emph{\begin{align*}
		b^{tyr,c}_\sI
		:=&
		\frac{m_1}{m} \var(\tx^c)^{(-)}\cov(\tx^c, y^c_0) + \frac{m_0}{m} \var(\tx^c)^{(-)}\cov(\tx^c, y^c_1).
		\end{align*}}
\end{definition}

Next, to define the corresponding coefficient estimator, first let specification I$^c$ be as follows

\begin{align*}
\xx^c_{\sI}:= \left[\begin{matrix}
-1_{\scriptscriptstyle m} &   0_{\scriptscriptstyle m} & -\tx^c
\\ 0_{\scriptscriptstyle m}  & 1_{\scriptscriptstyle m} & \tx^c
\end{matrix}\right].
\end{align*}
Note $\xx^c_{\sI}$ has $l+1$ columns where $\xx_{\sI}$ has only $l$.

\begin{definition}[Tyranny of the minority with cluster totals coefficient estimator]\label{def.tyr.c.est} The ``tyranny of the minority with cluster totals coefficient estimator" for specification \emph{I} is
	\emph{\begin{align*}
		\left[ \begin{matrix}
		\widehat{a}_0 \\ \widehat{a}_1 \\
		\widehat{b}^{tyr,c}_\sI
		\end{matrix} \right]
		:=& \left({\xx^c_\sI}' \R^c \left(({\bpi^c})^{-1}- \I_{\scriptscriptstyle 2m}\right) \xx^c_\sI \right)^{-1} {\xx^c_\sI}' \R^c \left(({\bpi^c})^{-1}- \I_{\scriptscriptstyle 2m}\right) y^c.
		\end{align*}} where ${\bpi^c}$ is a $2m \times 2m$ matrix giving probabilities of assignment along the diagonal. 	
\end{definition}

To prove that $b^{tyr,c}_\sI$ is an optimal choice of coefficient for the fixed-coefficient generalized regression estimator, first define an equivalent coefficient for specification II.

\begin{definition}[Tyranny of the minority with cluster totals for specification II]\label{def.tyr.c.II}
	The ``tyranny of the minority with cluster totals coefficient" for specification \emph{II} is 
	\emph{\begin{align*}
		b^{tyr,c}_\II = \left[\begin{matrix}
		  \frac{m_1}{m}\var({\tx^c})^{-1}\cov({\tx^c}, y^c_0)+\frac{m_0}{m}\var({\tx^c})^{-1}\cov(\tx^c, y^c_1)
		\\ 		  \frac{m_1}{m}\var({\tx^c})^{-1}\cov( {\tx^c}, y^c_0)+\frac{m_0}{m}\var({\tx^c})^{-1}\cov({\tx^c}, y^c_1)
		\end{matrix} \right].
		\end{align*}}
\end{definition}

Comparing Definition \ref{def.tyr.c.II} to Definition \ref{def.tyr.c} reveals that the ``slope" coefficients are identical in the two specifications.  The implication is that $\xx^c_\sI b^{tyr,c}_\sI=\xx^c_\II b^{tyr,c}_\II$ and hence, the conjugate ATE estimators are algebraically equivalent. Therefore, if $b^{tyr,c}_\II$ is in the set of optimal choices for a fixed-coefficient in specification II, then $b^{tyr,c}_\sI$ must be among the optimal coefficients for the fixed-coefficient generalized regression estimator for specification I.

\begin{theorem}
	For cluster-randomized experiments with specification \emph{II}, the tyranny of the minority with cluster totals coefficient given in Definition \ref{def.tyr.c.II} is an optimal coefficient for the fixed-coefficient generalized regression estimator.
\end{theorem}
\begin{proof}
	Beginning with Lemma \ref{lemma.allbopt} and again arriving at equation (\ref{b.optgZ}), this time let
	\begin{align*}
	z = & \left[\begin{matrix}  0_{\scriptscriptstyle (k+1)} \\  \frac{m_1}{m}\var(\tx^c)^{-1}\cov(\tx^c, y^c_0)+\frac{m_0}{m}\var(\tx^c)^{-1}\cov(\tx^c, y^c_1)  \end{matrix} \right].
	\end{align*}
	The result follows.
\end{proof}
\begin{corollary}
	For cluster-randomized experiments with specification \emph{I}, the tyranny of the minority coefficient given in Definition \ref{def.tyr.c} is an optimal coefficient for the fixed-coefficient generalized regression estimator.
\end{corollary}

\begin{theorem}
	Under Assumptions 1 and 2, in a cluster-randomized design with specification \emph{I}, the tyranny of the minority with cluster totals coefficient estimator given in Definition \ref{def.tyr.c.est} has a conjugate ATE estimator that obtains asymptotic minimum variance within the class of generalized regression estimators. 
\end{theorem}
\begin{proofatend}
	\begin{align*}
	\E\left[\xx_\sI' \R \left( \bpi^{-1}-\I_{\scriptscriptstyle 2n} \right) \xx_\sI \right]=& \xx_\sI' \bpi \left( \bpi^{-1}-\I_{\scriptscriptstyle 2n}\right) \xx_\sI
	\\=& \xx_\sI' \left( \I_{\scriptscriptstyle 2n}-\bpi\right) \xx_\sI
	\\ =& \left[\begin{matrix}
	0' & 1' \\ -1' & 1'\\ -\x' & \x'
	\end{matrix}\right] \left( \I_{\scriptscriptstyle 2n}-\bpi\right) \left[\begin{matrix}
	0 & -1 & -\x \\ 1 & 1 & \x
	\end{matrix}\right] 
	\\ =& \left[\begin{matrix}
	0' & 1' \\ -1' & 1'\\ -\x' & \x'
	\end{matrix}\right] \left[\begin{matrix}
	0 & -\frac{n_1}{n}1 & -\frac{n_1}{n}\x \\ \frac{n_0}{n}1 & \frac{n_0}{n}1 & \frac{n_0}{n}\x
	\end{matrix}\right] 
	\\ =& \left[\begin{matrix}
	n_0 & n_0 & 0 \\ n_0 & n & 0
	\\ 0 & 0 & n \var(\x)
	\end{matrix}\right]
	\end{align*}
	and 
	\begin{align*}
	\E\left[\xx_\sI' \R \left( \bpi^{-1}-\I_{\scriptscriptstyle 2n} \right) y \right]=& \xx_\sI' \bpi \left( \bpi^{-1}-\I_{\scriptscriptstyle 2n}\right) y
	\\=& \xx_\sI' \left( \I_{\scriptscriptstyle 2n}-\bpi\right) \xx_\sI
	\\ =& \left[\begin{matrix}
	0' & 1' \\ -1' & 1'\\ -\x' & \x'
	\end{matrix}\right] \left( \I_{\scriptscriptstyle 2n}-\bpi\right) \left[\begin{matrix}
	-y_0  \\ y_1
	\end{matrix}\right] 
	\\ =& \left[\begin{matrix}
	0' & 1' \\ -1' & 1'\\ -\x' & \x'
	\end{matrix}\right] \left[\begin{matrix}
	-\frac{n_1}{n}y_0  \\ \frac{n_0}{n}y_1
	\end{matrix}\right] 
	\\ =& \left[\begin{matrix}
	n_0 \mu_{y_1} \\ n_1 \mu_{y_0} + n_0 \mu_{y_1} \\ n_1 \cov(\x, y_0)+ n_0 \cov(\x, y_1)
	\end{matrix}\right]
	\end{align*}
	so that 
	\begin{align*}
	\E\left[\xx_\sI' \R \left( \bpi^{-1}-\I_{\scriptscriptstyle 2n} \right) \xx_\sI \right]^{-1}&\E\left[\xx_\sI' \R \left( \bpi^{-1}-\I_{\scriptscriptstyle 2n} \right) y \right]
	\\=& \left[\begin{matrix}
	n_0 & n_0 & 0 \\ n_0 & n & 0
	\\ 0 & 0 & n \var(\x)
	\end{matrix}\right] ^{-1} \left[\begin{matrix}
	n_0 \mu_{y_1} \\ n_1 \mu_{y_0} + n_0 \mu_{y_1} \\ n_1 \cov(\x, y_0)+ n_0 \cov(\x, y_1)
	\end{matrix}\right]
	\\=& \left[\begin{matrix}
	nn_1^{-1}n_0^{-1} & - n_1^{-1} & 0 
	\\ -n_1^{-1} & n_1^{-1} & 0
	\\ 0 & 0 & n^{-1} \var(\x)^{-1}
	\end{matrix}\right] \left[\begin{matrix}
	n_0 \mu_{y_1} \\ n_1 \mu_{y_0} + n_0 \mu_{y_1} \\ n_1 \cov(\x, y_0)+ n_0 \cov(\x, y_1)
	\end{matrix}\right]
	\\=&\left[\begin{matrix}
	\delta \\ \mu_{y_0} \\ \frac{n_1}{n} \var(\x)^{-1}\cov(\x, y_0)+ \frac{n_0}{n} \var(\x)^{-1} \cov(\x, y_1)
	\end{matrix}\right].
	\end{align*}
	Thus, under suitable regularity conditions $\widehat{b}^{tyr}_\sI \rightarrow b^{tyr}_\sI$ so that $\widehat{\delta}^{\tr,tyr}_\sI$ is asymptotically optimal.
\end{proofatend}
\begin{proof}
	Provided in Appendix.
\end{proof}

	\subsection{A tighter variance bound for cluster randomized experiments}
	
	There are tighter bounds than that implied by Aronow and Samii's $\tilde{\dmat}^{\tas}$ for cluster randomized experiments. Suppose, for example, you had complete randomization of clusters and at least two clusters assigned to each arm. Then $\tilde{\dmat}$ associated with the variance bound implied by \cite{middletonaronow} could be written 
	\begin{align*}
	\tilde{\dmat}^c =\dmat +\left[\begin{array}{ccc}
	\bfI \left( \dmat_{01}=-1 \right) & \bfI \left( \dmat_{01}=-1 \right) 
	\\ \bfI \left( \dmat_{01}=-1 \right) & \bfI\left( \dmat_{01}=-1 \right)\end{array}\right]
	\end{align*}

	\begin{theorem}
		For cluster randomized experiments, $n^{-2} y'\tilde{\dmat}^c y$ represents an identified bound for $n^{-2}y'{\dmat}y$. In other words, for all $ y \in \mathds{R}^{\scriptscriptstyle 2n}$, $ n^{-2} y' \tilde{\dmat} y \leq  n^{-2} y' \tilde{\dmat}^c y$ and $n^{-2}y'\tilde{\dmat}^c y$ is identified. 
	\end{theorem}
	
	\begin{proof}
		The quantity we are trying to bound is $\phi=2 \sum_i \sum_j y_{1i} y_{0j} \bfI \left(c_i = c_j \right)$, where $c_{i}$ gives the cluster id for the $i^{th}$ unit. The additional terms in the bound implied by $\tilde{\dmat}^c$ are 
		\begin{align*}
		\sum_i \sum_j \left(y_{0i} y_{0j}+y_{1i} y_{1j}\right) \bfI \left(c_i = c_j \right)=&\sum_i \sum_j \left(y_{1i} y_{0j}+y_{1i} y_{0j}+\tau_i \tau_j \right) \bfI \left(c_i = c_j \right)\\ = & \phi + \sum_i \sum_j \tau_i \tau_j \bfI \left(c_i = c_j \right).
		\end{align*}
		The sum $\sum_i \sum_j \tau_i \tau_j \bfI \left(c_i = c_j \right)$ must be positive for each cluster.
	\end{proof}
	
	\begin{theorem}
		For cluster randomized experiments, if the sharp null holds then $n^{-2}y'{\dmat}y=n^{-2}y'\tilde{\dmat}^c y$.
	\end{theorem}
	
	\begin{theorem}
		$\tilde{\dmat}^c$ provides a tighter bound under complete randomization of clusters than $\tilde{\dmat}^\tas$ of associated with the Aronow-Samii bound.
	\end{theorem}	
	\begin{proof}
		The additional terms being added by the AS method to bound the variance are $  \sum_i  \left( y_{0i}^2 +y_{1i}^2 \right)$.  So the question is whether $\sum_i \sum_j \left(y_{0i} y_{0j}+y_{1i} y_{1j}\right) \bfI \left(c_i = c_j \right) <  \sum_i  \left( y_{0i}^2 +y_{1i}^2 \right)$. To simplify, for a single cluster, $g$, consider whether $\sum_{i\in g} \sum_{j\in g} y_{0i} y_{0j} \leq \sum_{i \in g}  y_{0i}^2$. This inequality holds by Jensen's inequality. 
	\end{proof}

\section{Simulations }\label{section.sims}

Simulations illustrate the potential for efficiency gain in a hypothetical cluster-randomized experiment. 

\subsection{Data generation}

A simulated data set was created with 1000 units and 100 clusters, 40 assigned to treatment.  To create a range of cluster sizes, cluster membership, $c_i$, was determined by the equation
\begin{align*}
c_i = \text{trunc}(1+100 \times ((i-.5)/1000)^{1.2}).
\end{align*}
for $i \in \{1,2, ... , 1000\}$. A table of cluster sizes can be seen in Table 1 below. The distribution of cluster sizes is right skewed.

\begin{table}[!ht] \centering {Table 1: Cluster Sizes} \\ \footnotesize
	\begin{tabular}{c c}
		cluster size & number \\ \hline
		8 & 13 \\
		9 & 41 \\
		10 & 21 \\
		11 & 10 \\
		12 & 6 \\
		13 & 3 \\
		14 & 3 \\
		16 & 2 \\
		22 & 1 \\ \hline
		total & 100 \\
	\end{tabular}
\end{table}

Data were generated as follows:
\begin{align*}
x_{ci}=& \alpha_c + \epsilon_{xi}
\\ y_{1ci}= y_{0ci} = & -\alpha_c + x^*_{ci} + n_c -0.025 n_c^2 +\epsilon_{ i}
\end{align*}
where $x_{ci}$ is a covariate for individual $i$ in cluster $c$, $\alpha_c$ is drawn from a standard normal at the cluster level, $\epsilon_{xi}$ drawn from a standard normal at the individual level, $\epsilon_{i}$ is drawn from $N(0,5)$ at the individual level, $n_c$ is number of units in cluster $c$, and $y_{0ci}$ and $y_{1ci}$ are the potential outcomes under control and treatment, respectively. Note that $y_{1ci}=y_{0ci}$, i.e., the sharp null holds. Random components were drawn independently from one another.\footnote{To give an idea about the relative contribution of $\epsilon_{i}$ to the overall variability of $y_{0ci}$ ($y_{1ci}$), regressing $y_{0ci}$ ($y_{1ci}$) on $x_{ci}$, $\overline{x}_{c}$, $n_c$ and $n_c^2$, yielded an $R^2$ of 0.173.}  A single finite population was generated using the above DGP and maintained across all simulations.\footnote{Using alternative random number seeds does not meaningfully change the results.}

\subsection{Competing estimators}

Competing Estimators:
\begin{enumerate} 
	\item {\it WLS/OLS.} The benchmark estimator is a generalized regression estimator using the WLS with $\bpi^{-1}$ weights coefficient. This is equivalent to OLS in this case since $\pi_{1i}$ are equal for all $i$ and specification II is used. 
	\item \emph{\hththt.} The generalized regression estimator with coefficient given in Definition \ref{def.3ht}. 
	\item \emph{\rara .} The generalized regression estimator with coefficient given in Definition \ref{TSopt.definition}. 
	\item {\it OLS with cluster totals.} OLS with cluster totals as described in Definition \ref{def.ols.est}.
\end{enumerate}

All estimators used specification II.  There were four $\x$ specifications:
\begin{enumerate} 
	\item {$x$.}
	\item { $x$, and $\overline{x}_c$} 
	\item {$x$, $\overline{x}_c$ and $n_c$} 
	\item {$x$, $\overline{x}_c$, $n_c$ and ${n_c^2}$} 
\end{enumerate}

\subsection{Results}

Figure \ref{sim.res} presents the results of the simulations. Clockwise from the top left, the subfigures present the MSE, squared SE, the percent reduction in MSE (relative to WLS/OLS) and bias squared.  From left to right on the $x$-axis are the four specifications.

Results suggest that the {\hththt} has relatively poor performance overall.  In particular the MSE is very high in part due to a substantial bias.  The WLS/OLS estimator performs relatively poorly in terms of MSE for the first two specifications.  For these specifications the regression adjusted {\rara} and OLS with cluster totals performs well, obtaining an MSE that is about 60\% lower than the benchmark, WLS/OLS. For the third specification ($x$, $\overline{x}$ and $n_c$) the {\rara} performs the best, with an MSE that is about 13\% lower than the benchmark WLS/OLS.  For the final specification ($x$, $\overline{x}$, $n_c$ and $n_c^2$) the {\rara} and WLS/OLS perform about equally well, though both show evidence of model-overfit, i.e., an increase in MSE over the third specification. 

\begin{figure} 
	  \includegraphics[width=6.5in]{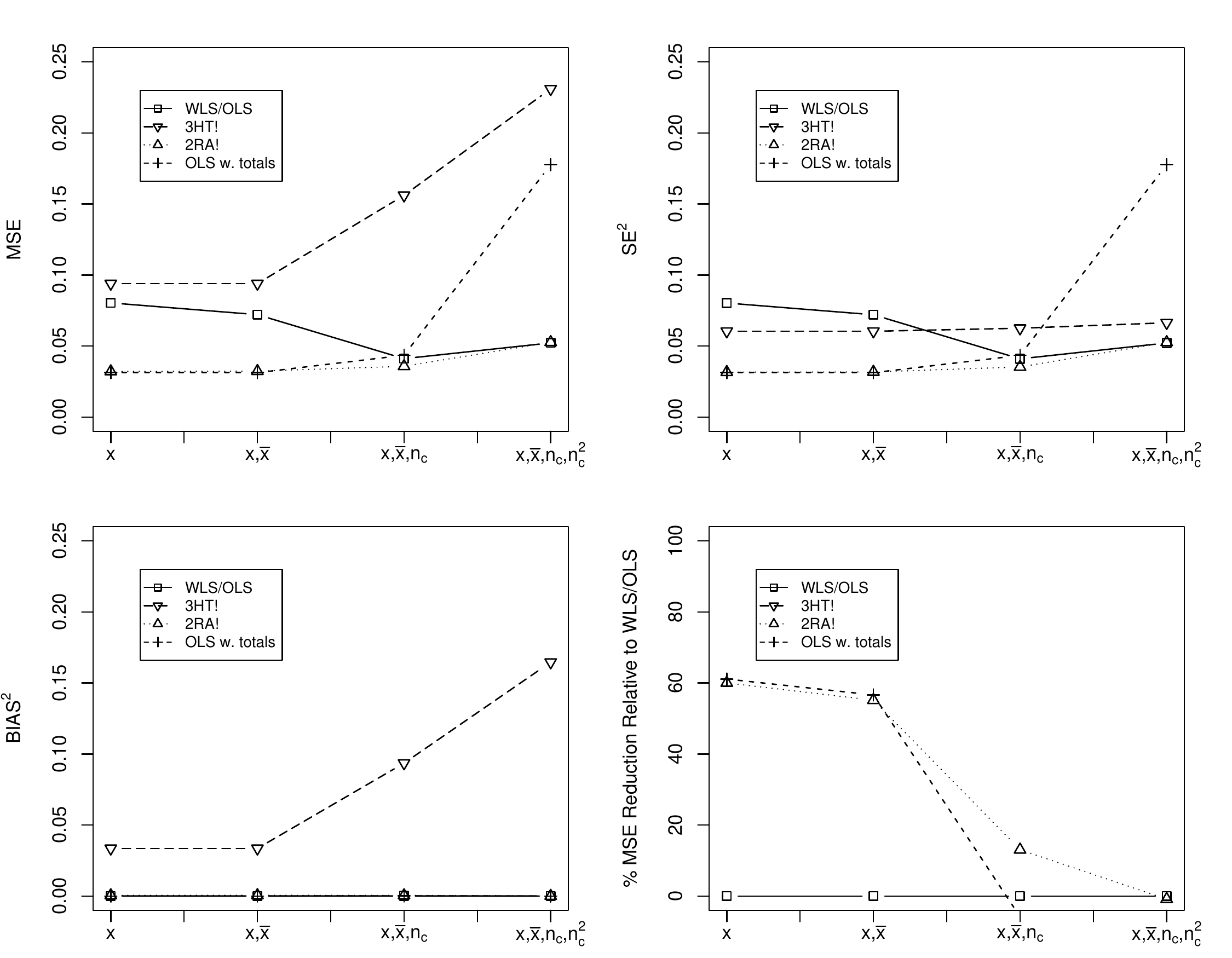}\caption{Results are from 5000 randomizations of a simulated cluster randomized experiment. Along the $x$-axis are four different covariate specifications, increasing in number of predictors from left to right. Each line depicts a coefficient estimation strategy. They are compared (clockwise from top left) in terms of MSE, SE$^2$, \% MSE Reduction and Bias$^2$.}\label{sim.res}
\end{figure}

\pagebreak
\begin{appendix}
	\section{Notation Index}
	\begin{tabular}{| C{2cm} | L{14cm} | } \hline &
		\\ $n$ & Number of units in the finite population in the experiment
		\\ &	\\  	$1_{\scriptscriptstyle 2n}$  & Length-$2n$ column vector of 1's. In matrix notation, serves as a replacement for the more common summation symbol, $\Sigma$
		\\ & \\$y_{0i}$, $y_{1i}$ & The control and treatment potential outcomes for the $i^{th}$ unit, respectively
		\\ & \\  $y_{0}$, $y_{1}$ & Length-$n$ vectors of control and treatment potential outcomes, respetively
		\\ & \\ $y$ & Length-$2n$ vector of all potential outcomes. The first $n$ elements are control potential outcomes multiplied by $-1$, followed by the treatment potential outcomes. Multiplication of control potential outcomes by $-1$ allows for the compact representation of the ATE as the sum of the elements of this vector divided by $n$
		\\ &	\\  	$\delta$ & Average treatment effect (ATE), the parameter of interest
		\\ & \\ $R_{0i}$, $R_{1i}$ & Random indicators of the $i^{th}$ unit's assignment to control and treatment, respectively
		\\ & \\ $R_{0}$, $R_{1}$ & Length-$n$ vectors of assignment indicators for control and treatment, respectively
		\\ & \\ $\R$ &$2n \times 2n$ diagonal matrix of assignment indicators. The first $n$ diagonal elements represent the control indicators, followed by $n$ treatment indicators
		\\ & \\ $\pi_{0i}$, $\pi_{1i}$ & For the $i^{th}$ unit, the probability of assignment to control and treatment, respectively
		\\ & \\ $\pi_{0}$, $\pi_{1}$ & Length-$n$ vectors of probabilities of assignment to control and treatment, respectively
		\\ &  \\	$\bpi$ & $2n\times2n$ diagonal matrix of assignment probabilities. The first $n$ diagonal elements give the control probabilities, followed by the treatment probabilities
		\\ & \\  $\pi_{0i0j}$, $\pi_{0i1j}$,   & Joint assignment probabilities for units $i$ and $j$. For example, $\pi_{1i0j}$ is the probability that 
		\\ $\pi_{1i0j}$, $\pi_{1i1j}$    & $i$ is in treatment and $j$ is in control
		\\ & \\ $\dmat$ & $2n\times2n$ ``design" matrix that gives the variance-covariance matrix of the vector $1'_{\scriptscriptstyle 2n}\bpi^{-1}\R$. Allows for compact representation of variance of HT estimators as a quadratic in matrix form
		\\ & \\ $\dmat_{00}$, $\dmat_{01}$, & The four $n\times n$ partitions of the matrix $\dmat$. For example, the top-right partition, $\dmat_{01}$, has
		\\ $\dmat_{10}$, $\dmat_{11}$ &  $i,j$ element $\frac{\pi_{0i1j}-\pi_{0i}\pi_{1j}}{\pi_{0i}\pi_{1j}}$
		\\ & \\ $\tilde{\dmat}$ & A modified version of $\dmat$ that allows for compact representation of a variance {\it bound} for HT estimators as a quadratic in matrix form. While the variance of the HT estimator is not identified, a variance bound may be
		\\ & \\ 
		\hline
	\end{tabular}
	\begin{tabular}{| C{2cm}| L{14cm} | } \hline &		
		\\ & \\ $\matp$ & $2n\times2n$ ``probability" matrix that gives the joint assignment probabilities
		\\ & \\ $\matp_{00}$, $\matp_{01}$, & The four $n\times n$ quadrants of the matrix $\matp$. For example, $\matp_{01}$ has $ij$ element $\pi_{0i1j}$
		\\ $\matp_{10}$, $\matp_{11}$ & 		
		\\ & \\ $\tilde{\matp}$ & A modified version of $\matp$ that replaces zeros with ones. Allows for division by $\tilde{\matp}$ without division-by-zero error 
		\\ & \\ $x_i$ & Length-$k$ vector of covariates associated with the $i^{th}$ unit	
		\\ &  \\ $\x$ & An $n \times k$ matrix of covariates	
		\\ &  \\ $\tx$ & An $n \times (k+1)$ matrix representing the concatenation of an intercept vector, $1_n$, and $\x$
		\\ &  \\ $\xx$ & A $2n \times l$ matrix of covariates. The first $n$ rows are multiplied by $-1$ to mirror the vector $y$. Represents an arbitrary specification
		\\ &  \\ $\xx_{\sI}$ & A $2n \times (k+2) $ matrix of covariates. The ``common slopes" specification. Elements in the first $n$ rows are multiplied by -1 to mirror the vector $y$
		\\ &  \\ $\xx_{\II}$ & A $2n \times (2k+2)$ matrix of covariates. The ``separate slopes" specification. Elements in the first $n$ rows are multiplied by -1 to mirror the vector $y$
		\\ & \\ 
		\hline
	\end{tabular}
	\pagebreak
\section{Supplementary Proofs}
\printproofs
\end{appendix}

\end{document}